\documentclass{article}

\usepackage{cmbright}
\usepackage{amssymb,amsmath,amsthm}
\usepackage{mathrsfs}
\usepackage[colorlinks,allcolors=blue]{hyperref}

\def\B{\mathscr B}
\def\C{\mathbb C}
\def\d{\mathrm{d}}
\def\dom{\mathcal D}
\def\g{\mathfrak g}
\def\GL{\mathrm{GL}}
\def\H{\mathcal H}
\def\L{\mathscr L}
\def\lone{\mathrm{L}^1}
\def\ltwo{\mathrm{L}^2}
\def\linf{\mathrm{L}^\infty}
\def\N{\mathbb N}
\def\Q{\mathbb Q}
\def\R{\mathbb R}
\def\S{\mathbb S}
\def\SO{\mathrm{SO}}
\def\SU{\mathrm{SU}}
\def\so{\mathfrak{so}}
\def\su{\mathfrak{su}}
\def\T{\mathbb T}
\def\U{\mathrm{U}}
\def\fraku{\mathfrak{u}}
\def\Z{\mathbb Z}

\def\Ad{\mathop{\mathrm{Ad}}\nolimits}
\def\diag{\mathop{\mathrm{diag}}\nolimits}
\def\e{\mathop{\mathrm{e}}\nolimits}
\DeclareMathOperator*{\essinf}{ess\hspace{1pt}inf}
\DeclareMathOperator*{\esssup}{ess\hspace{1pt}sup}
\def\det{\mathop{\mathrm{det}}\nolimits}
\def\Tr{\mathop{\mathrm{Tr}}\nolimits}
\DeclareMathOperator*{\slim}{s\hspace{0.1pt}-\hspace{0.1pt}lim}
\DeclareMathOperator*{\ulim}{u\hspace{0.1pt}-\hspace{0.1pt}lim}

\newtheorem{Theorem}{Theorem}[section]
\newtheorem{Remark}[Theorem]{Remark}
\newtheorem{Lemma}[Theorem]{Lemma}
\newtheorem{Corollary}[Theorem]{Corollary}
\newtheorem{Proposition}[Theorem]{Proposition}
\newtheorem{Example}[Theorem]{Example}

\begin{document}

%--------------------------------------------------------------------------------------
% Title
%--------------------------------------------------------------------------------------

\title{Degree, mixing, and absolutely continuous spectrum of cocycles with values in
compact Lie groups}

\author{R. Tiedra de Aldecoa\footnote{Supported by the Chilean Fondecyt Grant 1130168
and by the Iniciativa Cientifica Milenio ICM RC120002 ``Mathematical Physics'' from
the Chilean Ministry of Economy.}}

\date{\small}
\maketitle
\vspace{-1cm}

\begin{quote}
\emph{
\begin{itemize}
\item[] Facultad de Matem\'aticas, Pontificia Universidad Cat\'olica de Chile,\\
Av. Vicu\~na Mackenna 4860, Santiago, Chile
\item[] \emph{E-mail:} rtiedra@mat.puc.cl
\end{itemize}
}
\end{quote}

%--------------------------------------------------------------------------------------

\begin{abstract}
We consider skew products transformations
$$
T_\phi:X\times G\to X\times G,~~(x,g)\mapsto\big(F_1(x),g\;\!\phi(x)\big),
$$
where $X$ is a compact manifold with probability measure $\mu_X$, $G$ a compact Lie
group with Lie algebra $\frak g$, $F_1:X\to X$ the time-one map of a
measure-preserving flow, and $\phi\in C^1(X,G)$ a cocycle. Then, we define the
degree of $\phi$ as a suitable function $P_\phi M_\phi:X\to\frak g$, we show that the
degree of $\phi$ transforms in a natural way under Lie group homomorphisms and under
the relation of $C^1$-cohomology, and we explain how it generalises previous
definitions of degree of a cocycle. For each finite-dimensional irreducible unitary
representation $\pi$ of $G$, and $\frak g_\pi$ the Lie algebra of $\pi(G)$, we define
in an analogous way the degree of $\pi\circ\phi$ as a suitable function
$P_{\pi\circ\phi}M_{\pi\circ\phi}:X\to\frak g_\pi$. If $F_1$ is uniquely ergodic and
the functions $\pi\circ\phi$ are diagonal, or if $T_\phi$ is uniquely ergodic, then
the degree of $\phi$ reduces to a constant in $\frak g$ given by an integral (average)
over $X$. As a by-product, we obtain that there does not exist uniquely ergodic skew
products $T_\phi$ with nonzero degree if $G$ is a connected semisimple compact Lie
group.

Next, we show that $T_\phi$ is mixing in the orthocomplement of the kernel of
$P_{\pi\circ\phi}M_{\pi\circ\phi}$, and under some additional assumptions we show that
$U_\phi$ has purely absolutely continuous spectrum in that orthocomplement if
$(iP_{\pi\circ\phi}M_{\pi\circ\phi})^2$ is strictly positive. Summing up these
individual results for each representation $\pi$, one obtains a global result for the
mixing property and the absolutely continuous spectrum of $T_\phi$. As an application,
we present in more detail four explicit cases: when $G$ is a torus,
$G=\mathrm{SU}(2)$, $G=\mathrm{SO}(3,\mathbb R)$, and $G=\mathrm{U}(2)$. In each case,
the results we obtain are either new, or generalise previous results in particular
situations.

Our proofs rely on new results on positive commutator methods for unitary operators in
Hilbert spaces.
\end{abstract}


\smallskip


%--------------------------------------------------------------------------------------
\section{Introduction and main results}
\setcounter{equation}{0}
%--------------------------------------------------------------------------------------

Let $T_\phi:\T^2\to\T^2$ be the skew product transformation of the $2$-torus given by
\begin{equation}\label{eq_skew_torus_1}
T_\phi(x,z):=\big(x\e^{2\pi i\alpha},z\;\!\phi(x)\big),
\quad\phi\in C^1(\T,\T),~\alpha\in\R\setminus\Q.
\end{equation}
Since the works of H. Anzai \cite{Anz51} and H. Furstenberg \cite{Fur61}, it is known
that the scalar
\begin{equation}\label{eq_skew_torus_2}
M_{\phi,\star}
:=\int_\T\d\mu_\T(x)\left(\frac\d{\d t}\Big|_{t=0}\;\!
\phi\big(x\e^{2\pi it\alpha}\big)\right)\phi(x)^{-1}
\in i\;\!\R\quad\hbox{($\mu_\T$ Haar measure on $\T$),}
\end{equation}
called degree of the cocycle $\phi$, plays an important role in the study of the
ergodic properties and spectral properties of $T_\phi$. In particular, if
$M_{\phi,\star}\ne0$ and $\phi$ is absolutely continuous, then $T_\phi$ is uniquely
ergodic with respect to the product measure $\mu_\T\otimes\mu_\T$ and $T_\phi$ is
mixing in the orthocomplement of functions depending only on the first variable
\cite{ILR93}. Furthermore, if the derivative of $\phi$ satisfies some additional
regularity assumption (for instance bounded variation or Dini-continuity), then
$T_\phi$ has even Lebesgue spectrum with uniform countable multiplicity in the
orthocomplement of functions depending only on the first variable \cite{ILR93,Tie15}.

Some results on the degree of cocycles $\phi$ and its relation with the ergodic
properties and spectral properties of skew products $T_\phi$ have also been obtained
for more general cocycles $\phi:X\to G$ from a standard Borel space $X$ to a compact
Lie group $G$ (see for instance
\cite{Fra00,Fra00_2,Fra04,Fur61,GLL91,Iwa97_2,ILR93,Kar13,Tie15_2,Tie15_3,Tie15}).
However, in each case, some significant restriction is imposed on the cocycle $\phi:$
the base space $X$ is assumed to be a torus, the fiber group $G$ is assumed to be
particular compact Lie group (torus, $\SU(2)$, semisimple,\,\ldots), the cocycle
$\phi$ is assumed to be cohomologous in some way to a diagonal cocycle, etc. Our
purpose in this paper is to present a mathematical framework suited for a general
definition of the degree a cocycle $\phi$ with values in a compact Lie group $G$ and
the obtention of mixing and spectral properties of associated skew products $T_\phi$,
without imposing any of the aforementioned restriction.

Our results are the following. We consider skew products
$$
T_\phi:X\times G\to X\times G,~~(x,g)\mapsto\big(F_1(x),g\;\!\phi(x)\big),
$$
where $X$ is a compact manifold with probability measure $\mu_X$, $G$ a compact Lie
group with Lie algebra $\g$ and normalised Haar measure $\mu_G$, $F_1:X\to X$ the
time-one map of a $C^1$ measure-preserving flow $\{F_t\}_{t\in\R}$ on $X$, and
$\phi\in C^1(X,G)$ a cocycle. The base space $X$ and the diffeomorphism $F_1$ are
chosen in this way to have at disposition analogues of the derivative
$\frac\d{\d t}\big|_{t=0}$ and the flow
$\T\times\R\ni(x,t)\mapsto x\e^{2\pi it\alpha}\in\T$ appearing in the particular case
\eqref{eq_skew_torus_1}-\eqref{eq_skew_torus_2}.

Motivated by previous definitions of K. Fr{\c{a}}czek, N. Karaliolios and the author
\cite{Fra00_2,Kar13,Tie15_2}, we define the degree of the cocycle $\phi:X\to G$ as a
function $P_\phi M_\phi:X\to\g$ given in terms of the iterates of the Koopmann
operator $U_\phi$ of $T_\phi$ in the Hilbert space
$\H:=\ltwo(X\times G,\mu_X\otimes\mu_G)$ (Lemmas \ref{lemma_P_phi} and
\ref{lemma_dif_pi}). We show that the degree of $\phi$ transforms in a natural way
under Lie group homomorphisms and under the relation of $C^1$-cohomology (Lemma
\ref{lemma_inv_degrees}(a)-(b)) and we explain how it generalises previous definitions
of degree of a cocycle (Remark \ref{remark_degree}). In an analogous way, for each
finite-dimensional irreducible unitary representation $\pi$ of $G$, and $\g_\pi$ the
Lie algebra of $\pi(G)$, we define the degree of the cocycle $\pi\circ\phi:X\to\pi(G)$
as a function $P_{\pi\circ\phi}M_{\pi\circ\phi}:X\to\g_\pi$ given in terms of the
iterates of $U_\phi$ in the subspaces $\H^{(\pi)}_j\subset\H$ associated to $\pi$. We
show that the degree of $\pi\circ\phi$ is equal to the differential
$(\d\pi)_{e_G}\big((P_\phi M_\phi)(\;\!\cdot\;\!)\big)$ of $P_\phi M_\phi$ (Lemma
\ref{lemma_dif_pi}) and that it transforms in a natural way under Lie group
homomorphisms and under the relation of $C^1$-cohomology (Lemma
\ref{lemma_inv_degrees}(c)-(d)). Also, we present two particular cases where the
degree of $\phi$ takes a simple form. First, we show in Lemma \ref{lemma_diagonal}
that if $F_1$ is uniquely ergodic and the functions $\pi\circ\phi$ are diagonal (as in
the particular case \eqref{eq_skew_torus_1}-\eqref{eq_skew_torus_2}), then the degree
of $\phi$ reduces to the constant degree
\begin{equation}\label{form_int_1}
M_{\phi,\star}=\int_X\d\mu_X(x)\,(\L_Y\phi)(x)\phi(x)^{-1}\in\g,
\end{equation}
with $\L_Y$ the Lie derivative for the flow $\{F_t\}_{t\in\R}$. Second, we show in
Lemma \ref{lemma_ergodic} that if $T_\phi$ is uniquely ergodic, then the degree of
$\phi$ reduces to the constant degree
\begin{equation}\label{form_int_2}
M_{\phi,\star}
=\int_{X\times G}\d(\mu_X\otimes\mu_G)(x,g)\,
\Ad_g\big((\L_Y\phi)(x)\phi(x)^{-1}\big)\in\g,
\end{equation}
with $\Ad_g$ the inner automorphism of $\g$ induced by $g\in G$. The formulas
\eqref{form_int_1} and \eqref{form_int_2} generalise the formula
\eqref{eq_skew_torus_2} for the degree of $\phi$. As a by-product of formula
\eqref{form_int_2}, we obtain criteria for the non (unique) ergodicity of skew
products $T_\phi$ with $\phi\in C^1(X,G)$ (Theorem \ref{theorem_non_ergodic}), and we
show that if $G$ is connected and $T_\phi$ is uniquely ergodic, then the degree
belongs to the center $z(\g)$ of Lie algebra $\g$ (Remark \ref{remark_semisimple}(a)).
This implies in particular that there does not exist uniquely ergodic skew products
$T_\phi$ with $\phi\in C^1(X,G)$ and nonzero degree if $G$ is a connected semisimple
compact Lie group (Remark \ref{remark_semisimple}(b)).

Next, we present criteria for the mixing property of $U_\phi$ and the absolute
continuity of the spectrum of $U_\phi$ in the subspaces $\H^{(\pi)}_j$. Namely, we
show that $U_\phi$ is mixing in the orthocomplement of the kernel of
$P_{\pi\circ\phi}M_{\pi\circ\phi}$ in $\H^{(\pi)}_j$ (Theorem \ref{thm_mixing_D_pi}),
and under some additional assumptions we show that $U_\phi$ has purely absolutely
continuous spectrum in $\H^{(\pi)}_j$ if $(iP_{\pi\circ\phi}M_{\pi\circ\phi})^2$ is
strictly positive (Theorem \ref{thm_absolute_D_pi}). We also present a simplified
version of these results when $F_1$ uniquely ergodic and the functions $\pi\circ\phi$
are diagonal, and when $T_\phi$ is uniquely ergodic (Corollaries \ref{cor_mixing_D_pi}
and \ref{cor_absolute_D_pi}). Summing up these individual results in the subspaces
$\H^{(\pi)}_j$, one obtains a global result for the mixing property and the absolutely
continuous spectrum of $U_\phi$ in the whole Hilbert space $\H$.

As an application, we present in more detail four explicit cases: when $G=\T^{d'}$
($d'\in\N^*$), $G=\SU(2)$, $G=\SO(3,\R)$, and $G=\U(2)$. In the case $G=\T^{d'}$
(Section \ref{section_torus}), we prove the mixing property and the absolute
continuity of the spectrum of $U_\phi$ in the orthocomplement of functions depending
only on the first variable for cocycles with nonzero degree and suitable regularity
(see \eqref{eq_ac_torus_1} and \eqref{eq_ac_torus_2}). These results extend similar
results in the particular case $X\times G=\T^d\times\T^{d'}$, $d,d'\in\N^*$ 
\cite{Anz51,Cho87,Fra00,Iwa97_1,Iwa97_2,ILR93,Kus74,Tie15_2,Tie15_3}. In the case
$G=\SU(2)$ (Section \ref{section_SU(2)}), there is no uniquely ergodic skew product
$T_\phi$ with $\phi\in C^1\big(X,\SU(2)\big)$ and nonzero degree because $\SU(2)$ is a
connected semisimple compact Lie group. But, using a result of K. Fr{\c{a}}czek
\cite{Fra00_2}, we obtain that if $F_1$ is ergodic then any cocycle
$\phi\in C^1\big(X,\SU(2)\big)$ is cohomologous to a diagonal cocycle
$\delta:X\to\SU(2)$. Assuming $F_1$ uniquely ergodic and imposing suitable regularity
assumptions on $\phi$ and $\delta$, we then show the mixing property and the absolute
continuity of the spectrum of $U_\phi$ in appropriate subspaces of $\H$ (Theorems
\ref{thm_mixing_U_delta} and \ref{thm_absolute_U_delta}). These results extend similar
results of K. Fr{\c{a}}czek in the particular case $X\times G=\T\times\SU(2)$. In the
case $G=\SO(3,\R)$ (Section \ref{section_SO(3,R)}), there is no uniquely ergodic skew
product $T_\phi$ with $\phi\in C^1\big(X,\SO(3,\R)\big)$ and nonzero degree because
$\SO(3,\R)$ is a connected semisimple compact Lie group. But, we present the
representation theory needed to apply our results and explain how to show the mixing
property and the absolute continuity of the spectrum of $U_\phi$ in appropriate
subspaces of $\H$ under the assumption that $F_1$ is uniquely ergodic and the
functions $\pi\circ\phi$ diagonal. In the case $G=\U(2)$ (Section \ref{section_U(2)}),
we prove the mixing property and the absolute continuity of the spectrum of $U_\phi$
in appropriate subspaces of $\H$ for cocycles $\phi\in C^1\big(X,\U(2)\big)$ with
$T_\phi$ uniquely ergodic, nonzero degree, and suitable regularity (Theorems
\ref{thm_mixing_U(2)} and \ref{thm_ac_U(2)}). Also, using results of L. H. Eliasson
and X. Hou \cite{Eli02,Hou11}, we present an explicit example of skew-product $T_\phi$
satisfying the assumptions of Theorems \ref{thm_mixing_U(2)} and \ref{thm_ac_U(2)},
namely, a skew-product $T_\phi$ with $\phi\in C^\infty\big(\T^d;\U(2)\big)$
($d\in\N^*$), $T_\phi$ uniquely ergodic, and nonzero degree (Example
\ref{ex_skew_product_U(2)}). These results for $G=\U(2)$ are new.

To conclude, we give a brief description of the methods we use to prove our results.
In Section \ref{section_commutators}, we consider an abstract unitary operator $U$ in
a Hilbert space $\H$ and recall from \cite{RT15,Tie15_3} the following result: If
there exists an auxiliary self-adjoint operator $A$ in $\H$ such that $[A,U]$ is
bounded in some suitable sense and such that the strong limit
$$
D:=\slim_{N\to\infty}\frac1N\big[A,U^N\big]U^{-N}
$$
exists, then $U$ is mixing in $\ker(D)^\perp$. Building on this result, we determine
conditions under which the operator $U$ has in fact purely absolutely continuous
spectrum in the subspace $\ker(D)^\perp$. To do this, we introduce in Lemma
\ref{Lemma_A_B} and Proposition \ref{prop_conjugate} a new family of
(Ces\`aro-averaged) self-ajoint operators
$$
A_{D,N}:=\frac1N\sum_{n=0}^{N-1}U^n(AD+DA)U^{-n},\quad N\in\N^*.
$$
Then, we show that the operator $U$ satisfies suitable regularity conditions with
respect $A_{D,N}$ (Lemma \ref{U_C_1_A_DN}) and a positive commutator estimate with
$A_{D,N}$ (Proposition \ref{prop_Mourre}). Combining these results with known results
on commutator methods for unitary operators \cite{FRT13}, we obtain a general
criterion for the absolute continuity of the spectrum of $U$ in the subspace
$\ker(D)^\perp$ (Theorem \ref{thm_absolute}). Finally, we apply in Section
\ref{section_cocycles} these abstract results to the Koopmann operator $U_\phi$ in the
Hilbert spaces $\H^{(\pi)}_j$, with the auxiliary operator $A$ defined in terms of the
Lie derivative $\L_Y$ (Lemma \ref{lemma_A}). Doing so, we obtain from the abstract
theory our results on the mixing property and the absolutely continuous spectrum of
$U_\phi$ (Theorems \ref{thm_mixing_D_pi} and Theorem \ref{thm_absolute_D_pi}), and we
show that the degree $P_{\pi\circ\phi}M_{\pi\circ\phi}$ of $\pi\circ\phi$ is nothing
else but the value of the operator $D$ above in the particular case $U=U_\phi$ and $A$
given by $\L_Y$ (Lemma \ref{lemma_D_phi_pi}).\\

\noindent
{\bf Acknowledgements.} The author thanks E. Emsiz, K. Fr{\c{a}}czek, N. Karaliolios,
S. Richard and I. Zurri\'an for useful discussions.

%--------------------------------------------------------------------------------------
\section{Commutator methods for unitary operators}\label{section_commutators}
\setcounter{equation}{0}
%--------------------------------------------------------------------------------------

In this section, we present an abstract method for the construction of a conjugate
operator and the proof of a Mourre estimate for a general class of unitary operators.
To do so, we start by recalling some facts borrowed from
\cite{ABG96,FRT13,RT15,Tie15_3} on commutator methods for unitary operators and
regularity classes associated with them.

Let $\H$ be a Hilbert space with scalar product
$\langle\;\!\cdot\;\!,\;\!\cdot\;\!\rangle$ antilinear in the first argument, denote
by $\B(\H)$ the set of bounded linear operators on $\H$, and write $\|\cdot\|$ both
for the norm on $\H$ and the norm on $\B(\H)$. Let $A$ be a self-adjoint operator in
$\H$ with domain $\dom(A)$, and let $S\in\B(\H)$.
\begin{enumerate}
\item[(i)] We say that $S$ belongs to $C^{+0}(A)$, with notation $S\in C^{+0}(A)$, if
$S$ satisfies the Dini-type condition
$$
\int_0^1\frac{\d t}t\,\big\|\e^{-itA}S\e^{itA}-S\big\|<\infty.
$$
\item[(ii)] For any $k\in\N$, we say that $S$ belongs to $C^k(A)$, with notation
$S\in C^k(A)$, if the map
\begin{equation}\label{eq_group}
\R\ni t\mapsto\e^{-itA}S\e^{itA}\in\B(\H)
\end{equation}
is strongly of class $C^k$. In the case $k=1$, one has $S\in C^1(A)$ if and only if
the quadratic form
$$
\dom(A)\ni\varphi\mapsto\big\langle\varphi,SA\;\!\varphi\big\rangle
-\big\langle A\;\!\varphi,S\varphi\big\rangle\in\C
$$
is continuous for the topology induced by $\H$ on $\dom(A)$. We denote by $[S,A]$ the
bounded operator associated with the continuous extension of this form, or
equivalently $-i$ times the strong derivative of the function \eqref{eq_group} at
$t=0$.
\item[(iii)] We say that $S$ belongs to $C^{1+0}(A)$, with notation $S\in C^{1+0}(A)$,
if $S\in C^1(A)$ and $[A,S]\in C^{+0}(A)$.
\end{enumerate}

As banachisable topological vector spaces, the sets $C^2(A)$, $C^{1+0}(A)$, $C^1(A)$,
$C^{+0}(A)$ and $C^0(A)=\B(\H)$ satisfy the continuous inclusions
\cite[Sec.~5.2.4]{ABG96}
$$
C^2(A)\subset C^{1+0}(A)\subset C^1(A)\subset C^{+0}(A)\subset C^0(A).
$$

Let $U\in C^1(A)$ be a unitary operator with (complex) spectral measure
$E^U(\;\!\cdot\;\!)$ and spectrum $\sigma(U)\subset\S^1:=\{z\in\C\mid|z|=1\}$. If
there exist a Borel set $\Theta\subset\S^1$, a number $a>0$, and a compact operator
$K\in\B(\H)$ such that
\begin{equation}\label{eq_Mourre}
E^U(\Theta)U^{-1}[A,U]E^U(\Theta)\ge aE^U(\Theta)+K,
\end{equation}
one says that $U$ satisfies a Mourre estimate on $\Theta$ and that $A$ is a conjugate
operator for $U$ on $\Theta$. Also, one says that $U$ satisfies a strict Mourre
estimate on $\Theta$ if \eqref{eq_Mourre} holds with $K=0$. One of the consequences of
a Mourre estimate is to imply spectral properties for $U$ on $\Theta$. We recall these
spectral properties in the case $U\in C^{1+0}(A)$ (see
\cite[Thm.~2.7 \& Rem.~2.8]{FRT13}):

\begin{Theorem}[Spectral properties of $U$]\label{thm_spectral}
Let $U$ be a unitary operator in $\H$ and let $A$ be a self-adjoint operator in $\H$
with $U\in C^{1+0}(A)$. Suppose there exist an open set $\Theta\subset\S^1$, a number
$a>0$, and a compact operator $K$ in $\H$ such that
\begin{equation}\label{eq_Mourre_bis}
E^U(\Theta)U^{-1}[A,U]E^U(\Theta)\ge aE^U(\Theta)+K.
\end{equation}
Then, $U$ has at most finitely many eigenvalues in $\Theta$, each one of finite 
multiplicity, and $U$ has no singular continuous spectrum in $\Theta$. Furthermore, if
\eqref{eq_Mourre_bis} holds with $K=0$, then $U$ has no singular spectrum in $\Theta$.
\end{Theorem}

We also recall a result on the mixing property of $U$ in the case $U\in C^1(A)$ (see
\cite[Cor.~2.5 \& Rem.~2.6]{RT15} and \cite[Thm.~3.2]{Tie15_3}):

\begin{Theorem}[Mixing property of $U$]\label{thm_mixing}
Let $U$ be a unitary operator in $\H$ and let $A$ be a self-adjoint operator in $\H$
with $U\in C^1(A)$. Assume that the strong limit
$$
D:=\slim_{N\to\infty}\frac1N\big[A,U^N\big]U^{-N}
$$
exists. Then,
\begin{enumerate}
\item[(a)] $\lim_{N\to\infty}\big\langle\varphi,U^N\psi\big\rangle=0$ for each
$\varphi\in\ker(D)^\perp$ and $\psi\in\H$,
\item[(b)] $U|_{\ker(D)^\perp}$ has purely continuous spectrum.
\end{enumerate}
\end{Theorem}

If $U\in C^1(A)$, then $U^N\in C^1(A)$ for each $N\in\Z$ \cite[Prop.~5.1.5]{ABG96}.
Thus, all the operators $\frac1N\big[A,U^N\big]U^{-N}$ are bounded and self-adjoint,
and so is their strong limit $D$. Also, one has $[D,U^n]=0$ for each $n\in\Z$. Thus,
$\ker(D)^\perp$ is a reducing subspace for $U$, and the restriction
$U|_{\ker(D)^\perp}$ is a well-defined unitary operator \cite[Ex.~5.39(b)]{Wei80}.
Furthermore, if $U$ and $\widetilde U$ are unitarily equivalent unitary operators,
then their strong limits $D$ and $\widetilde D$ are unitarily equivalent. More
precisely:

\begin{Lemma}[Invariance of $D$]\label{lemma_inv_D}
Let $U$ be a unitary operator in $\H$ and let $A$ be a self-adjoint operator in $\H$
with $U\in C^1(A)$. Assume that the strong limit
$D=\slim_{N\to\infty}\frac1N\big[A,U^N\big]U^{-N}$ exists. Then, if
$\widetilde U:=VUV^{-1}$ with $V$ a unitary operator in $\H$ such that $V\in C^1(A)$,
the strong limit
$$
\widetilde D:=\slim_{N\to\infty}\frac1N\big[A,\widetilde U^N\big]\widetilde U^{-N}
$$
also exists, and satisfies $\widetilde D=VDV^{-1}$.
\end{Lemma}

\begin{proof}
Since $U$ and $V$ belong to $C^1(A)$, the operator $V^{-1}$ and the product
$\widetilde U=VUV^{-1}$ also belong to $C^1(A)$ (see \cite[Prop.~5.1.5 \&
5.1.6(a)]{ABG96}). Thus, the commutators $[A,VU^NV^{-1}]$, $[A,V^{-1}]$, $[A,U^N]$ and
$[A,V]$ are well-defined, and
\begin{align*}
\widetilde D
&=\slim_{N\to\infty}\frac1N\big[A,\widetilde U^N\big]\widetilde U^{-N}\\
&=\slim_{N\to\infty}\frac1N\big[A,VU^NV^{-1}\big]VU^{-N}V^{-1}\\
&=\slim_{N\to\infty}\frac1N\big(VU^N\big[A,V^{-1}\big]+V\big[A,U^N\big]V^{-1}
+[A,V]U^NV^{-1}\big)VU^{-N}V^{-1}\\
&=0+V\left(\slim_{N\to\infty}\frac1N\big[A,U^N\big]U^{-N}\right)V^{-1}+0\\
&=VDV^{-1}.
\end{align*}
\end{proof}

In the rest of the section, we build on Theorem \ref{thm_mixing} and determine
conditions under which the unitary operator $U$ of Theorem \ref{thm_mixing} has purely
absolutely continuous spectrum in $\ker(D)^\perp$. To do so, we construct a conjugate
operator for $U$ using the operators $A$ and $D$ of Theorem \ref{thm_mixing}. We start
with a simple, but useful, lemma on self-adjoint operators:

\begin{Lemma}\label{Lemma_A_B}
Let $A$ be a self-adjoint operator in $\H$ and let $B=B^*\in\B(\H)$ with
$B\in C^1(A)$.
\begin{enumerate}
\item[(a)] The operator
$$
\mathscr A_B\varphi:=(AB+BA)\varphi,\quad\varphi\in\dom(A),
$$
is essentially self-adjoint in $\H$, and its closure $A_B:=\overline{\mathscr A_B}$
has domain
$$
\dom(A_B)=\big\{\varphi\in\H\mid B\varphi\in\dom(A)\big\}.
$$
\item[(b)] If $C\in\B(\H)$ satisfies $C\in C^1(A)$ and $[B,C]=0$, then $C\in C^1(A_B)$
with $[C,A_B]=[C,A]B+B[C,A]$.
\end{enumerate}
\end{Lemma}

\begin{proof}
(a) The proof is inspired from that of \cite[Lemma~7.2.15]{ABG96}. Since
$B\in C^1(A)$, one has $B\;\!\dom(A)\subset\dom(A)$, and thus $\mathscr A_B$ is a
well-defined symmetric operator on $\dom(A)$. Furthermore, one has for
$\varphi\in\dom(A)$
$$
\mathscr A_B\varphi=(AB+BA)\varphi=\big([A,B]+2BA\big)\varphi.
$$
Therefore, the identity $[A,B]^*=[B,A]$ \cite[Lemma~7.2.15]{ABG96} and the usual
properties of the adjoint \cite[Thm.~4.19(b) \& 4.20(c)]{Wei80} imply
$$
(\mathscr A_B)^*=\big([A,B]+2BA\big)^*=[B,A]+2AB
\quad\hbox{with}\quad
\dom\big((\mathscr A_B)^*\big)
=\dom(AB)
=\big\{\varphi\in\H\mid B\varphi\in\dom(A)\big\}.
$$
Since $\dom(A)$ is a core for $AB$ \cite[p.~297-298]{ABG96}, it follows that $\dom(A)$
is a also core for $(\mathscr A_B)^*$. So, for each
$\varphi,\psi\in\dom\big((\mathscr A_B)^*\big)$ there exist
$\{\varphi_n\},\{\psi_n\}\subset\dom(A)$ such that
$$
\lim_n\|\varphi-\varphi_n\|_{\dom((\mathscr A_B)^*)}
=\lim_n\|\psi-\psi_n\|_{\dom((\mathscr A_B)^*)}
=0,
$$
and thus
\begin{align*}
\big\langle\varphi,(\mathscr A_B)^*\psi\big\rangle
&=\lim_n\lim_m\big\langle\varphi_n,\big([B,A]+2AB\big)\psi_m\big\rangle\\
&=\lim_n\lim_m\big\langle\big([B,A]^*+2BA\big)\varphi_n,\psi_m\big\rangle\\
&=\lim_n\lim_m\big\langle\big([A,B]+2BA\big)\varphi_n,\psi_m\big\rangle\\
&=\lim_n\lim_m\big\langle\big([B,A]+2AB\big)\varphi_n,\psi_m\big\rangle\\
&=\big\langle(\mathscr A_B)^*\varphi,\psi\big\rangle.
\end{align*}
Therefore, $(\mathscr A_B)^*$ is symmetric, and thus $\mathscr A_B$ is essentially
self-adjoint with $\overline{\mathscr A_B}=(\mathscr A_B)^*$ and
$$
\dom(\overline{\mathscr A_B})
=\dom\big((\mathscr A_B)^*\big)
=\dom(AB)
=\big\{\varphi\in\H\mid B\varphi\in\dom(A)\big\}.
$$

(b) Let $\varphi\in\dom(A)$. Then, the inclusions $B,C\in C^1(A)$ and the relation
$[B,C]=0$ imply that
\begin{align*}
\big\langle\varphi,CA_B\varphi\big\rangle
-\big\langle A_B\varphi,C\varphi\big\rangle
&=\big\langle\varphi,C(AB+BA)\varphi\big\rangle
-\big\langle\varphi,(AB+BA)C\varphi\big\rangle\\
&=\big\langle\varphi,\big([C,A]B+B[C,A]\big)\varphi\big\rangle.
\end{align*}
Since $\dom(A)$ is a core for $A_B$, it follows that $C\in C^1(A_B)$ with
$[C,A_B]=[C,A]B+B[C,A]$.
\end{proof}

Using Lemma \ref{Lemma_A_B}, we can define and prove the self-adjointness of a
suitable conjugate operator for $U$. This conjugate operator is an amelioration of the
conjugate operator introduced in \cite[Sec.~4]{FRT13} taking into account the
existence of the limit $D$.

\begin{Proposition}[Conjugate operator for $U$]\label{prop_conjugate}
Let $U$ be a unitary operator in $\H$ and let $A$ be a self-adjoint operator in $\H$
with $U\in C^1(A)$. Assume that the strong limit
$D=\slim_{N\to\infty}\frac1N\big[A,U^N\big]U^{-N}$ exists and satisfies $D\in C^1(A)$.
Then, the operator
$$
A_{D,N}\;\!\varphi:=\frac1N\sum_{n=0}^{N-1}U^nA_DU^{-n}\varphi,
\quad N\in\N^*,~\varphi\in\dom(A_{D,N}):=\dom(A_D),
$$
is self-adjoint in $\H$, and $\dom(A)$ is a core for $A_{D,N}$.
\end{Proposition}

\begin{proof}
$D$ is a bounded self-adjoint operator in $\H$ with $D\in C^1(A)$. So, Lemma
\ref{Lemma_A_B}(a) implies that $A_D$ is self-adjoint on
$
\dom(A_D)=\{\varphi\in\H\mid D\varphi\in\dom(A)\},
$
and that $\dom(A)$ is a core for $A_D$. Furthermore, since $U^n\in C^1(A)$ and
$[D,U^n]=0$ for each $n\in\Z$, we know from Lemma \ref{Lemma_A_B}(b) that
$U^n\in C^1(A_D)$ for each $n\in\Z$. Thus, $U^{-n}\;\!\dom(A_D)\subset\dom(A_D)$ for
each $n\in\Z$, and $A_{D,N}$ is a well-defined symmetric operator on $\dom(A_D)$.
Finally, we have for $\varphi\in\dom(A_D)$
\begin{equation}\label{eq_A_DN}
A_{D,N}\varphi
=\frac1N\sum_{n=0}^{N-1}U^nA_DU^{-n}\varphi
=A_D\varphi+\frac1N\sum_{n=0}^{N-1}[U^n,A_D]U^{-n}\varphi
\quad\hbox{with}\quad\frac1N\sum_{n=0}^{N-1}[U^n,A_D]U^{-n}\in\B(\H).
\end{equation}
Therefore, $A_{D,N}$ is self-adjoint on $\dom(A_{D,N})=\dom(A_D)$ and $\dom(A)$ is a
core for $A_{D,N}$.
\end{proof}

\begin{Lemma}\label{U_C_1_A_DN}
Let $U$ be a unitary operator in $\H$ and let $A$ be a self-adjoint operator in $\H$
with $U\in C^1(A)$. Assume that the strong limit
$D=\slim_{N\to\infty}\frac1N\big[A,U^N\big]U^{-N}$ exists and satisfies $D\in C^1(A)$.
Then, for each $N\in\N^*$,
\begin{enumerate}
\item[(a)] $U\in C^1(A_{D,N})$ with
$\big[A_{D,N},U\big]=\frac1N\sum_{n=0}^{N-1}U^n\big([A,U]D+D[A,U]\big)U^{-n}$,
\item[(b)] if $[A,U]\in C^{+0}(A_D)$, then $U\in C^{1+0}(A_{D,N})$,
\item[(c)] if $[A,U]\in C^1(A)$ and $[A,D]=0$, then $U\in C^2(A_{D,N})$.
\end{enumerate}
\end{Lemma}

In view of the results (b) and (c), we wonder if the following is also true: If
$[A,U]\in C^{+0}(A)$ and $[A,D]=0$, then we have $U\in C^{1+0}(A_{D,N})$\,? Answering
this question in the affirmative could allow to relax some regularity assumptions in
applications.

\begin{proof}
(a) Let $\varphi\in\dom(A)$. Then, the inclusion $U\in C^1(A)$ and the relation
$[D,U]=0$ imply that
\begin{align*}
\big\langle A_{D,N}\varphi,U\varphi\big\rangle
-\big\langle\varphi,UA_{D,N}\varphi\big\rangle
&=\frac1N\sum_{n=0}^{N-1}\big(\big\langle U^nA_DU^{-n}\varphi,U\varphi\big\rangle
-\big\langle\varphi,UU^nA_DU^{-n}\varphi\big\rangle\big)\\
&=\frac1N\sum_{n=0}^{N-1}\big(\big\langle\varphi,U^n(AD+DA)U^{-n}U\varphi\big\rangle
-\big\langle\varphi,UU^n(AD+DA)U^{-n}\varphi\big\rangle\big)\\
&=\frac1N\sum_{n=0}^{N-1}
\big\langle\varphi,U^n\big([A,U]D+D[A,U]\big)U^{-n}\varphi\big\rangle.
\end{align*}
Since $\dom(A)$ is a core for $A_{D,N}$, it follows that $U\in C^1(A_{D,N})$ with
$$
\big[A_{D,N},U\big]=\frac1N\sum_{n=0}^{N-1}U^n\big([A,U]D+D[A,U]\big)U^{-n}.
$$

(b) We know from point (a) that $U\in C^1(A_{D,N})$ with
$\big[A_{D,N},U\big]=\frac1N\sum_{n=0}^{N-1}U^n\big([A,U]D+D[A,U]\big)U^{-n}$. Therefore, it
is sufficient to show for $n\in\{0,\ldots,N-1\}$ that
$$
\int_0^1\frac{\d t}t\,\big\|\e^{-itA_{D,N}}B_n\e^{itA_{D,N}}-B_n\big\|<\infty
\quad\hbox{with}\quad
B_n:=U^n\big([A,U]D+D[A,U]\big)U^{-n}.
$$
But, for each $t\in\R$ and $\varphi\in\dom(A_D)$ we obtain from \eqref{eq_A_DN}
\begin{align*}
\e^{itA_{D,N}}\varphi-\e^{itA_D}\varphi
&=\int_0^t\d s\,\frac\d{\d s}\big(\e^{isA_{D,N}}\e^{-isA_D}-1\big)\e^{itA_D}\varphi\\
&=\int_0^t\d s\,\e^{isA_{D,N}}\frac iN\sum_{n=0}^{N-1}[U^n,A_D]U^{-n}
\e^{i(t-s)A_D}\varphi.
\end{align*}
So, there exists $C_t\in\B(\H)$ with $\|C_t\|\le{\rm Const.}\;\!|t|$ such that
$\e^{itA_{D,N}}=\e^{itA_D}+C_t$, and thus
\begin{equation}\label{eq_integral}
\int_0^1\frac{\d t}t\,\big\|\e^{-itA_{D,N}}B_n\e^{itA_{D,N}}-B_n\big\|
\le{\rm Const}+\int_0^1\frac{\d t}t\,\big\|\e^{-itA_D}B_n\e^{itA_D}-B_n\big\|.
\end{equation}
Now, the integral $\int_0^1\frac{\d t}t\,\big\|\e^{-itA_D}S\e^{itA_D}-S\big\|$ is
finite for $S=U^n$, $S=[A,U]$, $S=D$ and $S=U^{-n}$ due to the assumptions and Lemma
\ref{Lemma_A_B}(b). So, the integral on the r.h.s. of \eqref{eq_integral} is also
finite (see \cite[Prop.~5.2.3(b)]{ABG96}), and thus the claim is proved.

(c) We know from point (a) that $U\in C^1(A_{D,N})$ with
$\big[A_{D,N},U\big]=\frac1N\sum_{n=0}^{N-1}U^n\big([A,U]D+D[A,U]\big)U^{-n}$. Therefore, it
is sufficient to show that $D\in C^1(A_{D,N})$ and $[A,U]\in C^1(A_{D,N})$. But, we
know from \eqref{eq_A_DN} that
$$
A_{D,N}
=\frac1N\sum_{n=0}^{N-1}U^nA_DU^{-n}
=A_D+\frac1N\sum_{n=0}^{N-1}[U^n,A_D]U^{-n}
\quad\hbox{with}\quad\frac1N\sum_{n=0}^{N-1}[U^n,A_D]U^{-n}\in\B(\H).
$$
So, $C^1(A_D)\subset C^1(A_{D,N})$, and thus it is sufficient to show that
$D\in C^1(A_D)$ and $[A,U]\in C^1(A_D)$. Now, since $D\in C^1(A)$, the inclusion
$D\in C^1(A_D)$ follows directly from Lemma \ref{Lemma_A_B}(b). To show that
$[A,U]\in C^1(A_D)$, let $\varphi\in\dom(A)$. Then, the inclusions $U,D\in C^1(A)$ and
the relations $[A,D]=[U,D]=0$ imply
$$
\big(D[A,U]-[A,U]D\big)\varphi
=\big(DAU-DUA-AUD+UAD\big)\varphi
=0.
$$
So, $\big[D,[A,U]\big]=0$ due to the density of $\dom(A)$ in $\H$. In consequence, the
operator $[A,U]\in\B(\H)$ satisfies $[A,U]\in C^1(A)$ and $\big[D,[A,U]\big]=0$, and
thus $[A,U]\in C^1(A_D)$ due to Lemma \ref{Lemma_A_B}(b).
\end{proof}

In the next proposition, we prove a strict Mourre estimate for $U$ in the subspace
$\ker(D)^\perp$. The fact that we obtain a Mourre estimate for $U$ with no compact
term is consistent with the result of Theorem \ref{thm_mixing}(b) on the absence of
point spectrum of $U$ in $\ker(D)^\perp$. The uniform limit in the next proposition
(and in the sequel) refers to the limit in the topology of $\B(\H)$.

\begin{Proposition}[Strict mourre estimate for $U$]\label{prop_Mourre}
Let $U$ be a unitary operator in $\H$ and let $A$ be a self-adjoint operator in $\H$
with $U\in C^1(A)$. Assume that the uniform limit
$D=\ulim_{N\to\infty}\frac1N\big[A,U^N\big]U^{-N}$ exists and satisfies $D\in C^1(A)$.
Then, for each $\varepsilon>0$ there exists $N_\varepsilon\in\N^*$ such that
$$
U^{-1}\big[A_{D,N},U\big]\ge2D^2-\varepsilon\quad\hbox{for $N\ge N_\varepsilon$.}
$$
In particular, if there exist an open set $\Theta\subset\S^1$ and $c_\Theta>0$ such
that $D^2E^U(\Theta)\ge c_\Theta E^U(\Theta)$, then there exist $N\in\N^*$ and $a>0$
such that
$$
E^U(\Theta)U^{-1}\big[A_{D,N},U\big]E^U(\Theta)\ge aE^U(\Theta).
$$
\end{Proposition}

\begin{proof}
The limit $D$ exists in the uniform topology. Therefore, the limit $D$ also exists in
the strong topology, and Lemma \ref{U_C_1_A_DN}(a) and the relation $[D,U]=0$ hold.
Using these facts and the notation
$$
D_N
:=\frac1N\big[A,U^N\big]U^{-N}
=\frac1N\sum_{n=0}^{N-1}U^n\big([A,U]U^{-1}\big)U^{-n},
\quad N\in\N^*,
$$
we obtain
\begin{align*}
U^{-1}\big[A_{D,N},U\big]
&=U^{-1}\left(\frac1N\sum_{n=0}^{N-1}U^n
\big([A,U]U^{-1}D+D[A,U]U^{-1}\big)U^{-n}\right)U\\
&=U^{-1}\big(D_ND+DD_N\big)U\\
&=2D^2+U^{-1}\big((D_N-D)D+D(D_N-D)\big)U.
\end{align*}
Now, since $D\in\B(\H)$ and $\lim_{N\to\infty}\|D_N-D\|=0$, there exists for each
$\varepsilon>0$ an index $N_\varepsilon\in\N^*$ such that
$$
\big\|\big(D_N-D\big)D+D\big(D_N-D\big)\big\|\le\varepsilon
\quad\hbox{for $N\ge N_\varepsilon$.}
$$
It follows that
\begin{equation}\label{eq_epsilon}
U^{-1}\big[A_{D,N},U\big]
=2D^2-\varepsilon+U^{-1}\big(\varepsilon+\big(D_N-D\big)D+D\big(D_N-D\big)\big)U
\ge2D^2-\varepsilon
\quad\hbox{for $N\ge N_\varepsilon$.}
\end{equation}

To show the second claim, we note that if there exist an open set $\Theta\subset\S^1$
and $c_\Theta>0$ such that $D^2E^U(\Theta)\ge c_\Theta E^U(\Theta)$, then we obtain
from \eqref{eq_epsilon} that
$$
E^U(\Theta)U^{-1}\big[A_{D,N},U\big]E^U(\Theta)
\ge E^U(\Theta)\big(2D^2-\varepsilon\big)E^U(\Theta)
\ge\big(2c_\Theta-\varepsilon\big)E^U(\Theta)
\quad\hbox{for $N\ge N_\varepsilon$.}
$$
Since $\varepsilon>0$ is arbitrary, this implies that there exist $N\in\N^*$ and $a>0$
such that
$$
E^U(\Theta)U^{-1}\big[A_{D,N},U\big]E^U(\Theta)\ge aE^U(\Theta).
$$
\end{proof}

We conclude this section with a criterion for the absolute continuity of the spectrum
of $U$ in the subspace $\ker(D)^\perp$.

\begin{Theorem}[Absolutely continuous spectrum of $U$]\label{thm_absolute}
Let $U$ be a unitary operator in $\H$ and let $A$ be a self-adjoint operator in $\H$
with $U\in C^1(A)$. Assume that the uniform limit
$D=\ulim_{N\to\infty}\frac1N\big[A,U^N\big]U^{-N}$ exists, that $D\in C^1(A)$, that
$[A,U]\in C^{+0}(A_D)$, and that there exist an open set $\Theta\subset\S^1$ and
$c_\Theta>0$ such that
$$
D^2E^U(\Theta)\ge c_\Theta E^U(\Theta).
$$
Then, $U$ has purely absolutely continuous spectrum in $\Theta\cap\sigma(U)$.
\end{Theorem}

\begin{proof}
Since the limit $D$ exists in the uniform topology, the limit $D$ also exists in the
strong topology. So, Lemma \ref{U_C_1_A_DN}(b) applies and $U\in C^{1+0}(A_{D,N})$.
Moreover, we know from Proposition \ref{prop_Mourre} that there exist $N\in\N^*$ and
$a>0$ such that
$$
E^U(\Theta)U^{-1}\big[A_{D,N},U\big]E^U(\Theta)\ge aE^U(\Theta).
$$
Therefore, it follows by Theorem \ref{thm_spectral} that $U$ has purely absolutely
continuous spectrum in $\Theta\cap\sigma(U)$.
\end{proof}

%--------------------------------------------------------------------------------------
\section{Cocycles with values in compact Lie groups}\label{section_cocycles}
\setcounter{equation}{0}
%--------------------------------------------------------------------------------------

In this section, we apply the theory of Section \ref{section_commutators} to obtain
general results on the degree, the strong mixing property, and the absolutely
continuous spectrum of skew products transformations associated to cocycles with
values in compact Lie groups. We start with the definition of the skew products
transformations that we will consider.

Let $X$ be a smooth compact second countable Hausdorff manifold with Borel probability
measure $\mu_X$, and let $\{F_t\}_{t\in\R}$ be a $C^1$ measure-preserving flow on $X$.
The family of composition operators $\{V_t\}_{t\in\R}$ given by
$$
V_t\;\!\varphi:=\varphi\circ F_t,\quad\varphi\in\ltwo(X,\mu_X),
$$
defines a strongly continuous one-parameter unitary group satisfying
$V_t\;\!C^1(X)\subset C^1(X)$ for each $t\in\R$. Thus, Nelson's criterion
\cite[Thm.~VIII.10]{RS80} implies that the generator of the group $\{V_t\}_{t\in\R}$
$$
H\varphi:=\slim_{t\to0}it^{-1}(V_t-1)\varphi,
\quad\varphi\in\dom(H):=\left\{\varphi\in\ltwo(X,\mu_X)
\mid\lim_{t\to0}|t|^{-1}\big\|(V_t-1)\varphi\big\|_{\ltwo(X,\mu_X)}<\infty\right\},
$$
is essentially self-adjoint on $C^1(X)$ and given by
$$
H\varphi:=i\L_Y\varphi,\quad\varphi\in C^1(X),
$$
with $Y$ the $C^0$ vector field of the flow $\{F_t\}_{t\in\R}$ and $\L_Y$ the
corresponding Lie derivative.

Let $G$ be a compact Lie group with normalised Haar measure $\mu_G$ and identity
$e_G$. Then, each measurable function $\phi:X\to G$ induces a measurable cocycle
$X\times\Z\ni(x,n)\mapsto\phi^{(n)}(x)\in G$ over $F_1$ given by
$$
\phi^{(n)}(x):=
\begin{cases}
\phi(x)(\phi\circ F_1)(x)\cdots(\phi\circ F_{n-1})(x) & \hbox{if }n\ge1\\
\hfil e_G & \hbox{if }n=0\\
\hfil\big(\phi^{(-n)}\circ F_n\big)(x)^{-1} & \hbox{if }n\le-1.
\end{cases}
$$
One thus calls cocycle any measurable function $\phi:X\to G$. The skew product
associated to $\phi$ is the transformation
$$
T_\phi:X\times G\to X\times G,~~(x,g)\mapsto\big(F_1(x),g\;\!\phi(x)\big).
$$
$T_\phi$ is an invertible automorphism of the measure space
$(X\times G,\mu_X\otimes\mu_G)$ with iterates
\begin{equation}\label{eq_T^n}
T_\phi^n(x,g)=\big(F_n(x),g\;\!\phi^{(n)}(x)\big),
\quad n\in\Z,~g\in G,~\hbox{$\mu_X$-almost every $x\in X$.}
\end{equation}
The corresponding Koopman operator
$$
U_\phi\;\!\psi
:=\psi\circ T_\phi,\quad\psi\in\H:=\ltwo(X\times G,\mu_X\otimes\mu_G),
$$
is a unitary operator in $\H$.

\begin{Remark}\label{remark_cohomologous}
Two cocycles $\phi,\delta:X\to G$ are said to be cohomologous if there exists a
measurable function $\zeta:X\to G$, called transfer function, such that
$$
\phi(x)=\zeta(x)^{-1}\;\!\delta(x)\;\!(\zeta\circ F_1)(x)
\quad\hbox{ for $\mu_X$-almost every $x\in X$.}
$$
In such a case, the map $\iota:X\times G\to X\times G$ given by
$$
\iota(x,g):=\big(x,g\;\!\zeta(x)\big),
\quad g\in G,~\hbox{$\mu_X$-almost every $x\in X$,}
$$
is a metrical isomorphism of $T_\phi$ and $T_\delta$ \cite[Thm.~10.2.1]{CFS82}, the
operator
$$
S_\iota:\H\to\H,~~\psi\mapsto\psi\circ\iota,
$$
is unitary, and the Koopman operators $U_\phi$ and $U_\delta$ are unitarily equivalent
with unitary equivalence given by $S_\iota U_\delta(S_\iota)^{-1}=U_\phi$
\cite[Sec.~1.7]{CFS82}.
\end{Remark}

Let $\widehat G$ be the set of (equivalence classes of) finite-dimensional irreducible
unitary representations (IUR) of $G$. Then, each representation $\pi\in\widehat G$ is
a $C^\infty$ group homomorphism from $G$ to the unitary group $\U(d_\pi)$ of degree
$d_\pi:=\dim(\pi)$, and Peter-Weyl's theorem implies that the set of all matrix
elements $\{\pi_{jk}\}_{j,k=1}^{d_\pi}$ of all representations $\pi\in\widehat G$
forms an orthogonal basis of $\ltwo(G,\mu_G)$ with orthogonality relation
\cite[Cor.~4.10(b)]{Kna02}
\begin{equation}\label{eq_orthogonality}
\big\langle\pi_{jk},\pi_{j'k'}\big\rangle_{\ltwo(G,\mu_G)}
=(\d_\pi)^{-1}\;\!\delta_{jj'}\;\!\delta_{kk'},\quad j,j',k,k'\in\{1,\ldots,d_\pi\}.
\end{equation}
Accordingly, one has the orthogonal decomposition
\begin{equation}\label{eq_decompo}
\H=\bigoplus_{\pi\in\widehat G}\,\bigoplus_{j=1}^{d_\pi}\H^{(\pi)}_j
\quad\hbox{with}\quad
\H^{(\pi)}_j:=\bigoplus_{k=1}^{d_\pi}\ltwo(X,\mu_X)\otimes\{\pi_{jk}\}.
\end{equation}
A direct calculation shows that the operator $U_\phi$ is reduced by the decomposition
\eqref{eq_decompo}, with restriction $U_{\phi,\pi,j}:=U_\phi\big|_{\H^{(\pi)}_j}$
given by
$$
U_{\phi,\pi,j}\sum_{k=1}^{d_\pi}\varphi_k\otimes\pi_{jk}
=\sum_{k,\ell=1}^{d_\pi}\big(\varphi_k\circ F_1\big)\big(\pi_{\ell k}\circ\phi\big)
\otimes\pi_{j\ell},\quad\varphi_k\in\ltwo(X,\mu_X).
$$
This, together with \eqref{eq_T^n}, implies that
\begin{equation}\label{eq_U_n}
\big(U_{\phi,\pi,j}\big)^n\sum_{k=1}^{d_\pi}\varphi_k\otimes\pi_{jk}
=\sum_{k,\ell=1}^{d_\pi}\big(\varphi_k\circ F_n\big)
\big(\pi_{\ell k}\circ\phi^{(n)}\big)\otimes\pi_{j\ell},
\quad n\in\Z,~\varphi_k\in\ltwo(X,\mu_X).
\end{equation}

In the lemma below, we generalise the definition of a unitary operator introduced by
K. Fr{\c{a}}czek in the particular case $X=\T$ and $G=\SU(2)$ (see
\cite[Sec.~2]{Fra00_2}). For this, we have to recall some basic facts on Lie groups
and fix some notations. Since $G$ is a compact Lie group $G$, there exists an
injective $C^\infty$ group homomorphism $\pi_*:G\to\U(n)$ for some $n\in\N^*$
\cite[Cor.~4.22]{Kna02}. The Lie algebra $\g$ of $G$ is a real vector space of
dimension $\dim(\g)=\dim(G)$ which supports an $\Ad$-invariant scalar product
$\langle\;\!\cdot\;\!,\;\!\cdot\;\!\rangle_\g:\g\times\g\to\C$
\cite[Prop.~4.24]{Kna02}. That is, if $g\in G$ and $\Ad_g$ is the inner automorphism
of $\g$ given by
$$
\Ad_g:\g\to\g,~~Z\mapsto gZg^{-1},
$$
then we have
\begin{equation}\label{eq_invariance}
\big\langle\Ad_gZ_1,\Ad_gZ_2\big\rangle_\g=\big\langle Z_1,Z_2\big\rangle_\g
\quad\hbox{for all $Z_1,Z_2\in\g$.}
\end{equation}
We equip $\g$ with the topology induced by the norm $\|\cdot\|_\g$ associated to the
scalar product $\langle\;\!\cdot\;\!,\;\!\cdot\;\!\rangle_\g$ and with the
corresponding Borel $\sigma$-algebra, so that $\g$ is a topological measurable space.
Finally, we write $\ltwo(X,\g)$ for the Hilbert space of measurable functions
$f_1,f_2:X\to\g$ with scalar product
$$
\big\langle f_1,f_2\big\rangle_{\ltwo(X,\g)}
:=\int_x\d\mu_X(x)\,\big\langle f_1(x),f_2(x)\big\rangle_\g.
$$

\begin{Lemma}\label{lemma_W}
The operator $W_\phi:\ltwo(X,\g)\to\ltwo(X,\g)$ given by
$$
\big(W_\phi f\big)(x):=\Ad_{\phi(x)}(f\circ F_1)(x),
\quad f\in\ltwo(X,\g),~\hbox{$\mu_X$-almost every $x\in X$,}
$$
is unitary in $\ltwo(X,\g)$ and satisfies
\begin{equation}\label{eq_W^n}
\big(W_\phi^nf\big)(x):=\Ad_{\phi^{(n)}(x)}(f\circ F_n)(x),
\quad n\in\Z,~f\in\ltwo(X,\g),~\hbox{$\mu_X$-almost every $x\in X$.}
\end{equation}
\end{Lemma}

\begin{proof}
For each $n\in\Z$, let $W_n:\ltwo(X,\g)\to\ltwo(X,\g)$ be the operator given by
$$
\big(W_nf\big)(x):=\Ad_{\phi^{(n)}(x)}(f\circ F_n)(x),
\quad f\in\ltwo(X,\g),~\hbox{$\mu_X$-almost every $x\in X$.}
$$
Then, \eqref{eq_invariance} and the fact that $\{F_t\}_{t\in\R}$ preserves the measure
$\mu_X$ imply that
$$
\big\|W_nf\big\|_{\ltwo(X,\g)}^2
=\int_X\d\mu_X(x)\,\big\|\Ad_{\phi^{(n)}(x)}(f\circ F_n)(x)\big\|_\g^2
=\int_X\d\mu_X(x)\,\big\|(f\circ F_n)(x)\big\|_\g^2
=\|f\|_{\ltwo(X,\g)}^2.\label{eq_isometric}
$$
Thus, $W_n$ is isometric. Moreover, direct calculations using the definition of
$\phi^{(n)}$ show that
$$
W_nW_{-n}=W_{-n}W_n=1.
$$
So, $W_n$ is unitary with inverse $W_{-n}$.

It remains to show that $W_\phi^n=W_n$. We show it by induction on $n$ in the case
$n\in\N$ (the case $n\in-\N$ is analogous). For $n=0$ and $n=1$, we verify directly
that $W_\phi^0=W_0$ and $W_\phi^1=W_1$. For $n\ge2$, we make the induction hypothesis
that $W_\phi^{n-1}=W_{n-1}$. Then, we obtain for $f\in\ltwo(X,\g)$ and $\mu_X$-almost
every $x\in X$ that
\begin{align*}
\big(W_\phi^nf\big)(x)
&=\big(W_1W_{n-1}f\big)(x)\\
&=\phi(x)\big(W_{n-1}f\circ F_1\big)(x)\phi(x)^{-1}\\
&=\phi(x)\big(\phi^{(n-1)}\circ F_1\big)(x)(f\circ F_n)(x)
\big(\phi^{(n-1)}\circ F_1\big)(x)^{-1}\phi(x)^{-1}\\
&=\phi^{(n)}(x)(f\circ F_n)(x)\phi^{(n)}(x)^{-1}\\
&=\big(W_nf\big)(x),
\end{align*}
and thus $W_\phi^n=W_n$ for all $n\in\N$.
\end{proof}

In the next lemma, we collect some convergence results for the sequence
$\frac1N\sum_{n=0}^{N-1}W_\phi^n$, $N\in\N^*$, which generalise results of K.
Fr{\c{a}}czek in the particular case $X=\T$ and $G=\SU(2)$ (see
\cite[Lemmas~2.1 \& 2.2]{Fra00_2}). We write $\lone(X,\g)$ for the Banach space of
measurable functions $f:X\to\g$ with norm
$$
\|f\|_{\lone(X,\g)}:=\int_X\d\mu_X(x)\,\|f(x)\|_\g,
$$
and $\lone(X\times G,\g)$ for the Banach space of measurable functions
$\widetilde f:X\times G\to\g$ with norm
$$
\|\widetilde f\;\!\|_{\lone(X\times G,\g)}
:=\int_{X\times G}\d(\mu_X\otimes\mu_G)(x,g)\,\|\widetilde f(x,g)\|_\g.
$$

\begin{Lemma}\label{lemma_P_phi}
Let $P_\phi\in\B\big(\ltwo(X,\g)\big)$ be the orthogonal projection onto
$\ker(1-W_\phi)$.
\begin{enumerate}
\item[(a)] If $f\in\ltwo(X,\g)$, then the sequence $\frac1N\sum_{n=0}^{N-1}W_\phi^nf$
converges in $\ltwo(X,\g)$ to $P_\phi f$.
\item[(b)] If $f\in\lone(X,\g)$, then the sequence
$\frac1N\sum_{n=0}^{N-1}\big(W_\phi^nf\big)(x)$ converges in $\g$ for $\mu_X$-almost
every $x\in X$.
\item[(c)] If $f\in\ltwo(X,\g)$, then the sequence
$\frac1N\sum_{n=0}^{N-1}\big(W_\phi^nf\big)(x)$ converges in $\g$ to $(P_\phi f)(x)$
for $\mu_X$-almost every $x\in X$.
\end{enumerate}
\end{Lemma}

\begin{proof}
(a) We know from Lemma \ref{lemma_W} that $W_\phi$ is a unitary operator in
$\ltwo(X,\g)$. Therefore, von Neumann's ergodic theorem implies that for each
$f\in\ltwo(X,\g)$ the sequence $\frac1N\sum_{n=0}^{N-1}W_\phi^nf$ converges in
$\ltwo(X,\g)$ to $P_\phi f$.

(b) Let $f\in\lone(X,\g)$ and let $\widetilde f:X\times G\to\g$ be given by
$$
\widetilde f(x,g):=\Ad_gf(x),\quad g\in G,~\hbox{$\mu_X$-almost every $x\in X$.}
$$
Then, \eqref{eq_T^n} and \eqref{eq_W^n} imply for every $g\in G$ and $\mu_X$-almost
every $x\in X$
\begin{align}
\widetilde f\big(T_\phi^n(x,g)\big)
&=\widetilde f\big(F_n(x),g\;\!\phi^{(n)}(x)\big)\nonumber\\
&=\Ad_{g\;\!\phi^{(n)}(x)}f\big(F_n(x)\big)\nonumber\\
&=\Ad_g\Ad_{\phi^{(n)}(x)}(f\circ F_n)(x)\nonumber\\
&=\Ad_g\big(W_\phi^nf\big)(x).\label{eq_W_and_tilde}
\end{align}
Moreover, we have $\widetilde f\in\lone(X\times G,\g)$ because \eqref{eq_invariance}
and Tonnelli's theorem imply that
\begin{align*}
\big\|\widetilde f\;\!\big\|_{\lone(X\times G,\g)}
&=\int_{X\times G}\d(\mu_X\otimes\mu_G)(x,g)\,\big\|\Ad_gf(x)\big\|_\g\\
&=\int_G\d\mu_G(g)\int_X\d\mu_X(x)\,\|f(x)\|_\g\\
&=\|f\|_{\lone(X,\g)}.
\end{align*}
So, we can apply Birkhoff's pointwise ergodic theorem for Banach-valued functions
\cite[Thm.~4.2.1]{Kre85} to obtain that the sequence
$$
\frac1N\sum_{n=0}^{N-1}\widetilde f\big(T_\phi^n(x,g)\big)
=\Ad_g\left(\frac1N\sum_{n=0}^{N-1}\big(W_\phi^nf\big)(x)\right),\quad N\in\N^*,
$$
converges in $\g$ for $(\mu_X\otimes\mu_G)$-almost every $(x,g)\in X\times G$.
Therefore, there exists $g\in G$ such that
$\Ad_g\big(\frac1N\sum_{n=0}^{N-1}\big(W_\phi^nf\big)(x)\big)$ converges in $\g$ for
$\mu_X$-almost every $x\in X$, and thus
$\frac1N\sum_{n=0}^{N-1}\big(W_\phi^nf\big)(x)$ converges in $\g$ for $\mu_X$-almost
every $x\in X$.

(c) Let $f\in\ltwo(X,\g)$. Then, we know from point (a) that the sequence
$\frac1N\sum_{n=0}^{N-1}W_\phi^nf$ converges in $\ltwo(X,\g)$ to $P_\phi f$. Thus,
$\frac1N\sum_{n=0}^{N-1}W_\phi^nf$ converges in measure $\mu_X$ to $P_\phi f$. On
another hand, since $\ltwo(X,\g)\subset\lone(X,\g)$, we know from point (b) that
$\frac1N\sum_{n=0}^{N-1}\big(W_\phi^nf\big)(x)$ converges in $\g$ for $\mu_X$-almost
every $x\in X$. Therefore, it follows from the uniqueness of the limit under different
modes of convergence \cite[Prop.~1.5.7]{Tao11} that
$\frac1N\sum_{n=0}^{N-1}\big(W_\phi^nf\big)(x)$ converges in $\g$ to
$(P_\phi f)(x)$ for $\mu_X$-almost every $x\in X$.
\end{proof}

\begin{Remark}\label{remark_Pf}
Since $P_\phi\in\B\big(\ltwo(X,\g)\big)$ is the orthogonal projection onto
$\ker(1-W)$, the function $P_\phi f\in\ltwo(X,\g)$ of Lemma \ref{lemma_P_phi}(a) satisfies
$W_\phi^n(P_\phi f)=P_\phi f$ for each $n\in\Z$, which means that $P_\phi f$ is
invariant under the action of the unitary operator $W_\phi$. Due to \eqref{eq_W^n},
this implies that
$$
\Ad_{\phi^{(n)}(x)}\big(P_\phi f\circ F_n\big)(x)=(P_\phi f)(x)
$$
for every $n\in\Z$ and $\mu_X$-almost every $x\in X$. It follows from
\eqref{eq_invariance} that the function $\rho_{\phi,f}\in\ltwo(X,\mu_X)$ given by
$\rho_{\phi,f}(\;\!\cdot\;\!):=\|(P_\phi f)(\;\!\cdot\;\!)\|_\g$ satisfies for
$\mu_X$-almost every $x\in X$
$$
\big(\rho_{\phi,f}\circ F_1\big)(x)
=\big\|\Ad_{\phi(x)}\big(P_\phi f\circ F_1\big)(x)\big\|_\g
=\rho_{\phi,f}(x).
$$
In consequence, if $F_1$ is ergodic, then $\rho_{\phi,f}(x)$ is constant for
$\mu_X$-almost every $x\in X$. This fact (noted by K. Fr{\c{a}}czek in
\cite[Lemma~2.1]{Fra00_2}) will be useful in Section \ref{section_SU(2)} when we will
discuss the case of cocycles $\phi$ taking values in the group $\SU(2)$.
\end{Remark}

We now begin to apply the theory of Section \ref{section_commutators} to the operator
$U_{\phi,\pi,j}$ in the Hilbert space $\H^{(\pi)}_j$. We start by defining an
appropriate self-adjoint operator $A$ in $\H^{(\pi)}_j$ and give conditions so that
$U_{\phi,\pi,j}\in C^1(A)$. We write $\B(\C^{d_\pi})$ for the set of
$d_\pi\times d_\pi$ complex matrices equipped with the operator norm and we write
$\linf\big(X,\B(\C^{d_\pi})\big)$ for the Banach space of measurable functions
$f:X\to\B(\C^{d_\pi})$ with norm
$$
\|f\|_{\linf(X,\B(\C^{d_\pi}))}:=\esssup_{x\in X}\|f(x)\|_{\B(\C^{d_\pi})}.
$$

\begin{Lemma}\label{lemma_A}
The operator
$$
A\sum_{k=1}^{d_\pi}\varphi_k\otimes\pi_{jk}
:=\sum_{k=1}^{d_\pi}H\varphi_k\otimes\pi_{jk},
\quad\varphi_k\in C^1(X),
$$
is essentially self-adjoint in $\H^{(\pi)}_j$, and its closure (which we denote by the
same symbol) has domain
$$
\dom(A)=\bigoplus_{k=1}^{d_\pi}\dom(H)\otimes\{\pi_{jk}\}.
$$
Furthermore, if $\L_Y(\pi\circ\phi)\in\linf\big(X,\B(\C^{d_\pi})\big)$, then
$U_{\phi,\pi,j}\in C^1(A)$ with
$$
\big[A,U_{\phi,\pi,j}\big]=iM_{\pi\circ\phi}\;\!U_{\phi,\pi,j},
$$
where $M_{\pi\circ\phi}$ is the bounded matrix-valued multiplication operator in
$\H^{(\pi)}_j$ given by
$$
M_{\pi\circ\phi}\sum_{k=1}^{d_\pi}\varphi_k\otimes\pi_{jk}
:=\sum_{k,\ell=1}^{d_\pi}\big(\L_Y(\pi\circ\phi)\cdot(\pi\circ\phi)^{-1}\big)_{k\ell}
\;\!\varphi_\ell\otimes\pi_{jk},\quad\varphi_k\in\ltwo(X,\mu_X).
$$
\end{Lemma}

\begin{proof}
The claims can be shown as in \cite[Lemma~3.2]{Tie15_2} (just put all the coefficients
$a_k$ equal to $1$ in the analogue of $M_{\pi\circ\phi}$ in \cite{Tie15_2}).
\end{proof}

In the sequel, we write $T_xX$ for the tangent space of $X$ at $x\in X$, $\g_\pi$ for
the Lie algebra of $\pi(G)\subset\U(d_\pi)$ equipped with the ($\Ad$-invariant)
operator norm $\|\cdot\|_{\B(\C^{d_\pi})}$, $T_gG$ for the tangent space of $G$ at
$g\in G$, $T_{\pi(g)}\pi(G)$ for the tangent space of $\pi(G)$ at $\pi(g)\in\pi(G)$,
and
$$
(\d\pi)_g:T_gG\to T_{\pi(g)}\pi(G)
$$
for the differential of the map $\pi:G\to\pi(G)$ at $g\in G$. We recall from
\cite[Sec.~2.1.7-2.1.8]{BLU07} that, due to the standard isomorphisms
$T_{e_G}G\simeq\g$ and $T_{\pi(e_G)}\pi(G)\simeq\g_\pi$, the differential
$(\d\pi)_{e_G}$ induces a $\R$-linear (and thus continuous) map from $\g$ to $\g_\pi$
which we denote by the same symbol, that is,
$$
(\d\pi)_{e_G}:\g\to\g_\pi.
$$
Finally, we let $\ltwo(X,\g_\pi)$ be the Banach space of measurable functions
$f:X\to\g_\pi$ with norm
$$
\|f\|_{\ltwo(X,\g_\pi)}
:=\int_x\d\mu_X(x)\,\|f(x)\|_{\B(\C^{d_\pi})}^2,
$$
we define the operator $W_{\pi\circ\phi}:\ltwo(X,\g_\pi)\to\ltwo(X,\g_\pi)$ by
$$
\big(W_{\pi\circ\phi} f\big)(x):=\Ad_{(\pi\circ\phi)(x)}(f\circ F_1)(x),
\quad f\in\ltwo(X,\g_\pi),~\hbox{$\mu_X$-almost every $x\in X$,}
$$
and we note from Lemma \ref{lemma_W} (applied with $\phi$ replaced by $\pi\circ\phi$)
that $W_{\pi\circ\phi}$ is unitary in $\ltwo(X,\g_\pi)$, with iterates
$$
\big(W_{\pi\circ\phi}^nf\big)(x):=\Ad_{(\pi\circ\phi^{(n)})(x)}(f\circ F_n)(x),
\quad n\in\Z,~f\in\ltwo(X,\g_\pi),~\hbox{$\mu_X$-almost every $x\in X$.}
$$

\begin{Lemma}\label{lemma_dif_pi}
Assume that
$$
(\L_Y\phi)(x)=\frac\d{\d t}\Big|_{t=0}\;\!\phi\big(F_t(x)\big)\in T_{\phi(x)}G
$$
exists for $\mu_X$-almost every $x\in X$ and that
$M_\phi:=\L_Y\phi\cdot\phi^{-1}\in\ltwo(X,\g)$. Then, we have for $\mu_X$-almost every
$x\in X$
\begin{equation}\label{eq_d_pi_PM}
\lim_{N\to\infty}\left\|\frac1N\sum_{n=0}^{N-1}\big(W_{\pi\circ\phi}^n
M_{\pi\circ\phi}\big)(x)-(\d\pi)_{e_G}\big((P_\phi M_\phi)(x)\big)\right\|
_{\B(\C^{d_\pi})}=0.
\end{equation}
\end{Lemma}

\begin{proof}
We have for every $n\in\Z$ and $\mu_X$-almost every $x\in X$
\begin{align*}
(\d\pi)_{e_g}\left(\big(W_\phi^nM_\phi\big)(x)\right)
&=\frac\d{\d t}\Big|_{t=0}\;\!
\pi\big(\e^{t\Ad_{\phi^{(n)}(x)}(M_\phi\circ F_n)(x)}\big)\\
&=\frac\d{\d t}\Big|_{t=0}\;\!
\pi\left(\Ad_{\phi^{(n)}(x)}\e^{t(M_\phi\circ F_n)(x)}\right)\\
&=\Ad_{(\pi\circ\phi^{(n)})(x)}\frac\d{\d t}\Big|_{t=0}
\pi\big(\e^{t(M_\phi\circ F_n)(x)}\big)\\
&=\Ad_{(\pi\circ\phi^{(n)})(x)}(\d\pi)_{e_G}\big(\big(M_\phi\circ F_n\big)(x)\big)
\end{align*}
and
\begin{align}
(\d\pi)_{e_g}\big(M_\phi(x)\big)
&=(\d\pi)_{e_g}\left(\left(\frac\d{\d t}\Big|_{t=0}\phi\big(F_t(x)\big)\right)
\phi(x)^{-1}\right)\nonumber\\
&=\frac\d{\d t}\Big|_{t=0}\;\!\pi\big(\phi\big(F_t(x)\big)\phi(x)^{-1}\big)\nonumber\\
&=\big(\L_Y(\pi\circ\phi)\big)(x)\cdot\big(\pi\circ\phi\big)(x)^{-1}\nonumber\\
&=M_{\pi\circ\phi}(x).\label{eq_d_M_phi}
\end{align}
These relations, together with the $\R$-linearity of $(\d\pi)_{e_g}$ and
\eqref{eq_W^n}, imply for every $N\in\N^*$ and $\mu_X$-almost every $x\in X$ that
$$
\frac1N\sum_{n=0}^{N-1}\big(W_{\pi\circ\phi}^nM_{\pi\circ\phi}\big)(x)
=\frac1N\sum_{n=0}^{N-1}(\d\pi)_{e_g}\left(\big(W_\phi^nM_\phi\big)(x)\right)
=(\d\pi)_{e_g}\left(\frac1N\sum_{n=0}^{N-1}\big(W_\phi^nM_\phi\big)(x)\right).
$$
Therefore, using the continuity of $(\d\pi)_{e_g}$ and Lemma \ref{lemma_P_phi}(c), we
obtain for $\mu_X$-almost every $x\in X$
\begin{align*}
&\lim_{N\to\infty}\left\|\frac1N\sum_{n=0}^{N-1}\big(W_{\pi\circ\phi}^n
M_{\pi\circ\phi}\big)(x)-(\d\pi)_{e_G}\big((P_\phi M_\phi)(x)\big)
\right\|_{\B(\C^{d_\pi})}\\
&=\lim_{N\to\infty}\left\|(\d\pi)_{e_g}\left(\frac1N\sum_{n=0}^{N-1}
\big(W_\phi^nM_\phi\big)(x)-\big(P_\phi M_\phi\big)(x)\right)
\right\|_{\B(\C^{d_\pi})}\\
&\le{\rm Const.}\lim_{N\to\infty}\left\|\frac1N\sum_{n=0}^{N-1}
\big(W_\phi^nM_\phi\big)(x)-\big(P_\phi M_\phi\big)(x)\right\|_\g\\
&=0.
\end{align*}
\end{proof}

\begin{Remark}\label{rem_P_pi_phi}
If $\L_Y\phi$ exists $\mu_X$-almost everywhere and $M_\phi\in\ltwo(X,\g)$, then
\eqref{eq_d_M_phi} and the continuity of $(\d\pi)_{e_g}$ imply that
\begin{align*}
\big\|M_{\pi\circ\phi}\big\|_{\ltwo(X,\g_\pi)}^2
&=\int_x\d\mu_X(x)\,\big\|(\d\pi)_{e_g}\big(M_\phi(x)\big)\big\|_\g^2\\
&\le{\rm Const.}\int_x\d\mu_X(x)\,\big\|M_\phi(x)\big\|_\g^2\\
&={\rm Const.}\;\!\big\|M_\phi\big\|_{\ltwo(X,\g)}^2\\
&<\infty.
\end{align*}
Therefore, an application of Lemma \ref{lemma_P_phi}(c) with $\phi$ replaced by
$\pi\circ\phi$ implies that the sequence
$\frac1N\sum_{n=0}^{N-1}\big(W_{\pi\circ\phi}^nM_{\pi\circ\phi}\big)(x)$ converges in
$\g_\pi$ to $\big(P_{\pi\circ\phi}M_{\pi\circ\phi}\big)(x)$ for $\mu_X$-almost every
$x\in X$. This, together with Lemma \ref{lemma_dif_pi}, implies that
$$
P_{\pi\circ\phi}M_{\pi\circ\phi}=(\d\pi)_{e_G}\big((P_\phi M_\phi)(\;\!\cdot\;\!)\big)
\quad\hbox{$\mu_X$-almost everywhere.}
$$
\end{Remark}

The functions $P_\phi M_\phi\in\ltwo(X,\g)$ and
$P_{\pi\circ\phi}M_{\pi\circ\phi}\in\ltwo(X,\g_\pi)$ of Lemma \ref{lemma_dif_pi} and
Remark \ref{rem_P_pi_phi} transform in a natural way under Lie group homomorphisms and
under the relation of $C^1$-cohomology:

\begin{Lemma}[Invariance of $P_\phi M_\phi$ and $P_{\pi\circ\phi}M_{\pi\circ\phi}$]
\label{lemma_inv_degrees}
Assume that $\phi\in C^1(X,G)$.
\begin{enumerate}
\item[(a)] If $\phi=h\circ\delta$ with $h:G'\to G$ a Lie group homomorphism
and $\delta\in C^1(X,G')$, then
$$
P_\phi M_\phi=(\d h)_{e_{G'}}\big((P_\delta M_\delta)(\;\!\cdot\;\!)\big)
\quad\hbox{$\mu_X$-almost everywhere.}
$$
\item[(b)] If $\zeta,\delta\in C^1(X,G)$ are such that
\begin{equation}\label{eq_coho_C1}
\phi(x)=\zeta(x)^{-1}\;\!\delta(x)\;\!(\zeta\circ F_1)(x)\quad\hbox{for each $x\in X$,}
\end{equation}
then $P_\delta M_\delta=\Ad_\zeta(P_\phi M_\phi)$ $\mu_X$-almost everywhere.
In particular,
\begin{equation}\label{eq_equal_norms}
\big\|\big(P_\delta M_\delta\big)(x)\big\|_\g
=\big\|\big(P_\phi M_\phi\big)(x)\big\|_\g
\quad\hbox{for $\mu_X$-almost every $x\in X$.}
\end{equation}
\item[(c)] If $\phi=h\circ\delta$ with $h:G'\to G$ a Lie group homomorphism
and $\delta\in C^1(X,G')$, then
$$
P_{\pi\circ\phi}M_{\pi\circ\phi}
=\big(\d(\pi\circ h)\big)_{e_{G'}}\big((P_\delta M_\delta)(\;\!\cdot\;\!)\big)
\quad\hbox{$\mu_X$-almost everywhere.}
$$
\item[(d)] If $\zeta,\delta\in C^1(X,G)$ are such that
$$
\phi(x)=\zeta(x)^{-1}\;\!\delta(x)\;\!(\zeta\circ F_1)(x)\quad\hbox{for each $x\in X$,}
$$
then
$
P_{\pi\circ\delta}M_{\pi\circ\delta}
=\Ad_{\pi\circ\zeta}(P_{\pi\circ\phi}M_{\pi\circ\phi})
$
$\mu_X$-almost everywhere. In particular,
$$
\big\|\big(P_{\pi\circ\delta}M_{\pi\circ\delta}\big)(x)\big\|_{\B(\C^{d_\pi})}
=\big\|\big(P_{\pi\circ\phi}M_{\pi\circ\phi}\big)(x)\big\|_{\B(\C^{d_\pi})}
\quad\hbox{for $\mu_X$-almost every $x\in X$.}
$$
\end{enumerate}
\end{Lemma}

\begin{proof}
(a) Since $\phi\in C^1(X,G)$ and $\delta\in C^1(X,G')$, an application of Lemma \ref{lemma_P_phi} for $\phi$ and $\delta$ implies the existence of
the limits $P_\phi M_\phi$ and $P_\delta M_\delta$ $\mu_X$-almost everywhere.
Then, one shows that $P_\phi M_\phi=(\d h)_{e_G}\big((P_\delta M_\delta)(\;\!\cdot\;\!)\big)$
$\mu_X$-almost everywhere as in Lemma \ref{lemma_dif_pi} (just replace in Lemma
\ref{lemma_dif_pi} $G$ by $G'$, $\phi$ by $\delta$, $\pi$ by $h$, and the norm
$\|\cdot\|_{\B(\C^{d_\pi})}$ by a norm on the Lie algebra of $G'$).

(b)  Since $\phi,\delta\in C^1(X,G)$, an application of Lemma \ref{lemma_P_phi} for $\phi$
and $\delta$ implies the existence of the limits $P_\phi M_\phi$ and
$P_\delta M_\delta$ $\mu_X$-almost everywhere. Furthermore, the equation
\eqref{eq_coho_C1}, the product rule in Lie groups, and the equality
$\L_Y(\zeta\circ F_1)=(\L_Y\zeta)\circ F_1$ imply that
\begin{align*}
M_\phi
&=\L_Y\phi\cdot\phi^{-1}\\
&=-\zeta^{-1}\cdot\L_Y\zeta\cdot\zeta^{-1}\cdot\delta\cdot(\zeta\circ F_1)
\cdot\phi^{-1}+\zeta^{-1}\cdot\L_Y\delta\cdot(\zeta\circ F_1)\cdot\phi^{-1}
+\zeta^{-1}\cdot\delta\cdot\L_Y(\zeta\circ F_1)\cdot\phi^{-1}\\
&=-\Ad_{\zeta^{-1}}\big(M_\zeta-M_\delta-\Ad_\delta(M_\zeta\circ F_1)\big).
\end{align*}
This, together with the fact that
$P_\phi(\Ad_{\zeta^{-1}}f)=\Ad_{\zeta^{-1}}(P_\delta f)$ for any $f\in\ltwo(X,\g)$,
gives
$$
P_\phi M_\phi
=-\Ad_{\zeta^{-1}}\big(P_\delta M_\zeta-P_\delta M_\delta
-P_\delta\big(\Ad_\delta(M_\zeta\circ F_1)\big)\big).
$$
Now, a direct calculation implies that
\begin{align*}
P_\delta\big(\Ad_\delta(M_\zeta\circ F_1)\big)
&=\lim_{N\to\infty}\frac1N\sum_{n=0}^{N-1}\Ad_{\delta^{(n+1)}}(M_\zeta\circ F_{n+1})\\
&=\lim_{N\to\infty}\frac1N\sum_{n=0}^{N-1}\Ad_{\delta^{(n)}}(M_\zeta\circ F_n)\\
&=P_\delta M_\zeta.
\end{align*}
Thus $P_\phi M_\phi=\Ad_{\zeta^{-1}}P_\delta M_\delta$ $\mu_X$-almost everywhere, and
the claim is proved.

(c) Remark \ref{rem_P_pi_phi}, point (a), and the chain rule
$\big(\d(\pi\circ h)\big)_{e_{G'}}=(\d\pi)_{e_G}\circ(\d h)_{e_{G'}}$ imply that
$$
P_{\pi\circ\phi}M_{\pi\circ\phi}
=(\d\pi)_{e_G}\big((P_\phi M_\phi)(\;\!\cdot\;\!)\big)
=(\d\pi)_{e_G}\big((\d h)_{e_{G'}}\big((P_\delta M_\delta)(\;\!\cdot\;\!)\big)\big)
=\big(\d(\pi\circ h)\big)_{e_{G'}}\big((P_\delta M_\delta)(\;\!\cdot\;\!)\big)
$$
$\mu_X$-almost everywhere.

(d) Since $\phi,\delta\in C^1(X,G)$, an application of Remark \ref{rem_P_pi_phi} with
$\phi$ and $\delta$ implies that
$$
P_{\pi\circ\phi}M_{\pi\circ\phi}
=(\d\pi)_{e_G}\big((P_\phi M_\phi)(\;\!\cdot\;\!)\big)
\quad\hbox{and}\quad
P_{\pi\circ\delta}M_{\pi\circ\delta}
=(\d\pi)_{e_G}\big((P_\delta M_\delta)(\;\!\cdot\;\!)\big)
$$
$\mu_X$-almost everywhere. Therefore, it follows from point (b) that
$$
P_{\pi\circ\delta}M_{\pi\circ\delta}
=(\d\pi)_{e_G}\big(\Ad_\zeta(P_\phi M_\phi)(\;\!\cdot\;\!)\big)
=\Ad_{\pi\circ\zeta}(\d\pi)_{e_G}\big((P_\phi M_\phi)(\;\!\cdot\;\!)\big)
=\Ad_{\pi\circ\zeta}P_{\pi\circ\phi}M_{\pi\circ\phi}
$$
$\mu_X$-almost everywhere.
\end{proof}

Using Lemma \ref{lemma_dif_pi}, we can now prove that the strong limit
$$
D_{\phi,\pi}
:=\slim_{N\to\infty}\frac1N\big[A,\big(U_{\phi,\pi,j}\big)^N\big]
\big(U_{\phi,\pi,j}\big)^{-N}
$$ 
exists and is equal to the matrix-valued multiplication operator by the function
$i\;\!(\d\pi)_{e_G}\big((P_\phi M_\phi)(\;\!\cdot\;\!)\big)$ or equivalently the
function $iP_{\pi\circ\phi}M_{\pi\circ\phi}$ (this is the same due to Remark
\ref{rem_P_pi_phi}).

\begin{Lemma}[Existence of $D_{\phi,\pi}$]\label{lemma_D_phi_pi}
Assume that $\L_Y\phi$ exists $\mu_X$-almost everywhere, that $M_\phi\in\ltwo(X,\g)$,
and that $\L_Y(\pi\circ\phi)\in\linf\big(X,\B(\C^{d_\pi})\big)$. Then, the strong
limit
$$
D_{\phi,\pi}
=\slim_{N\to\infty}\frac1N\big[A,\big(U_{\phi,\pi,j}\big)^N\big]
\big(U_{\phi,\pi,j}\big)^{-N}
$$ 
exists and satisfies
$$
D_{\phi,\pi}\sum_{k=1}^{d_\pi}\varphi_k\otimes\pi_{jk}
=\sum_{k,\ell=1}^{d_\pi}i\;\!\big((\d\pi)_{e_G}
\big((P_\phi M_\phi)(\;\!\cdot\;\!)\big)\big)_{k\ell}
\;\!\varphi_\ell\otimes\pi_{jk},\quad\varphi_k\in\ltwo(X,\mu_X).
$$
\end{Lemma}

The assumptions of Lemma \ref{lemma_D_phi_pi} are satisfied for instance if
$\phi\in C^1(X,G)$.

\begin{proof}
We know from \eqref{eq_U_n} that
$$
\big(U_{\phi,\pi,j}\big)^n\sum_{k=1}^{d_\pi}\varphi_k\otimes\pi_{jk}
=\sum_{k,\ell=1}^{d_\pi}\big(\varphi_k\circ F_n\big)
\big(\pi_{\ell k}\circ\phi^{(n)}\big)\otimes\pi_{j\ell},
\quad n\in\Z,~\varphi_k\in\ltwo(X,\mu_X),
$$
and we know from Lemma \ref{lemma_A} that $U_{\phi,\pi,j}\in C^1(A)$ with
$[A,U_{\phi,\pi,j}]=iM_{\pi\circ\phi}U_{\phi,\pi,j}$. Therefore, we have for each
$N\in\N^*$ the equalities
\begin{align*}
\frac1N\big[A,\big(U_{\phi,\pi,j}\big)^N\big]\big(U_{\phi,\pi,j}\big)^{-N}
&=\frac1N\sum_{n=0}^{N-1}\big(U_{\phi,\pi,j}\big)^n\left(\big[A,U_{\phi,\pi,j}\big]
\big(U_{\phi,\pi,j}\big)^{-1}\right)\big(U_{\phi,\pi,j}\big)^{-n}\\
&=\frac1N\sum_{n=0}^{N-1}\big(U_{\phi,\pi,j}\big)^niM_{\pi\circ\phi}
\big(U_{\phi,\pi,j}\big)^{-n}\\
&=\frac iN\sum_{n=0}^{N-1}\big(\pi\circ\phi^{(n)}\big)
\big(M_{\pi\circ\phi}\circ F_n\big)\big(\pi\circ\phi^{(n)}\big)^{-1}\\
&=\frac iN\sum_{n=0}^{N-1}W_{\pi\circ\phi}^nM_{\pi\circ\phi},
\end{align*}
and the operators
$$
D_{\phi,\pi,N}:=\frac iN\sum_{n=0}^{N-1}W_{\pi\circ\phi}^nM_{\pi\circ\phi},
\quad N\in\N^*,
$$
are bounded matrix-valued multiplication operators in $\H^{(\pi)}_j$. On another hand,
Lemma \ref{lemma_dif_pi} implies that the sequence
$\frac1N\sum_{n=0}^{N-1}\big(W_{\pi\circ\phi}^nM_{\pi\circ\phi}\big)(x)$ converges in
$\B(\C^{d_\pi})$ to $(\d\pi)_{e_G}\big((P_\phi M_\phi)(x)\big)$ for $\mu_X$-almost
every $x\in X$. This, together with the assumption
$\L_Y(\pi\circ\phi)\in\linf\big(X,\B(\C^{d_\pi})\big)$, implies for $\mu_X$-almost
every $x\in X$ that
\begin{align*}
\big\|(\d\pi)_{e_G}\big((P_\phi M_\phi)(x)\big)\big\|_{\B(\C^{d_\pi})}
&=\lim_{N\to\infty}\left\|\frac1N\sum_{n=0}^{N-1}
\big(W_{\pi\circ\phi}^nM_{\pi\circ\phi}\big)(x)\right\|_{\B(\C^{d_\pi})}\\
&=\lim_{N\to\infty}\left\|\frac1N\sum_{n=0}^{N-1}\Ad_{(\pi\circ\phi^{(n)})(x)}
\big(M_{\pi\circ\phi}\circ F_n\big)(x)\right\|_{\B(\C^{d_\pi})}\\
&\le\lim_{N\to\infty}\frac1N\sum_{n=0}^{N-1}
\big\|\big(M_{\pi\circ\phi}\circ F_n\big)(x)\big\|_{\B(\C^{d_\pi})}\\
&\le\big\|M_{\pi\circ\phi}\big\|_{\linf(X,\B(\C^{d_\pi}))}\\
&<\infty.
\end{align*}
So, we have
$
(\d\pi)_{e_G}\big((P_\phi M_\phi)(\;\!\cdot\;\!)\big)
\in\linf\big(X,\B(\C^{d_\pi})\big)
$,
and thus the matrix-valued multiplication operator in $\H^{(\pi)}_j$ given by
$$
(\d\pi)_{e_G}\big((P_\phi M_\phi)(\;\!\cdot\;\!)\big)
\sum_{k=1}^{d_\pi}\varphi_k\otimes\pi_{jk}
=\sum_{k,\ell=1}^{d_\pi}
\big((\d\pi)_{e_G}\big((P_\phi M_\phi)(\;\!\cdot\;\!)\big)\big)_{k\ell}
\;\!\varphi_\ell\otimes\pi_{jk},\quad\varphi_k\in\ltwo(X,\mu_X),
$$
is bounded (we use the notation
$(\d\pi)_{e_G}\big((P_\phi M_\phi)(\;\!\cdot\;\!)\big)$ both for the element of
$\linf\big(X,\B(\C^{d_\pi})\big)$ and the corresponding multiplication operator in
$\H^{(\pi)}_j$).

To prove the claim, we have to show that the strong limit
$\slim_{N\to\infty}D_{\phi,\pi,N}$ exists and is equal to
$i\;\!(\d\pi)_{e_G}\big((P_\phi M_\phi)(\;\!\cdot\;\!)\big)$. A direct calculation
using the orthogonality relation \eqref{eq_orthogonality} and the notations
$$
E_{\phi,\pi,N}
:=\frac iN\sum_{n=0}^{N-1}W_{\pi\circ\phi}^nM_{\pi\circ\phi}
-i\;\!(\d\pi)_{e_G}\big((P_\phi M_\phi)(\;\!\cdot\;\!)\big)
\quad\hbox{and}\quad
\varphi:=(\varphi_1,\ldots,\varphi_{d_\pi})^\top
$$
gives for $\sum_{k=1}^{d_\pi}\varphi_k\otimes\pi_{jk}\in\H^{(\pi)}_j$
\begin{align}
&\left\|\left(D_{\phi,\pi,N}
-i\;\!(\d\pi)_{e_G}\big((P_\phi M_\phi)(\;\!\cdot\;\!)\big)\right)
\sum_{k=1}^{d_\pi}\varphi_k\otimes\pi_{jk}\right\|_{\H^{(\pi)}_j}^2\nonumber\\
&=(d_\pi)^{-1}\sum_{k=1}^{d_\pi}
\left\|\sum_{\ell=1}^{d_\pi}(E_{\phi,\pi,N})_{k\ell}\;\!\varphi_\ell\right\|
_{\ltwo(X,\mu_X)}^2\nonumber\\
&=(d_\pi)^{-1}\sum_{k=1}^{d_\pi}
\big\|\big(E_{\phi,\pi,N}\;\!\varphi\big)_k\big\|_{\ltwo(X,\mu_X)}^2\nonumber\\
&=(d_\pi)^{-1}\int_X\d\mu_X(x)\,
\big\|E_{\phi,\pi,N}(x)\;\!\varphi(x)\big\|_{\C^{\d_\pi}}^2\nonumber\\
&\le(d_\pi)^{-1}\int_X\d\mu_X(x)\,\big\|E_{\phi,\pi,N}(x)\big\|_{\B(\C^{\d_\pi})}^2
\|\varphi(x)\|_{\C^{\d_\pi}}^2.\label{eq_Lebesgue}
\end{align}
Furthermore, the inclusions
$\frac1N\sum_{n=0}^{N-1}W_{\pi\circ\phi}^nM_{\pi\circ\phi}\in\linf\big(X,\B(\C^{d_\pi})\big)$
and
$
(\d\pi)_{e_G}\big((P_\phi M_\phi)(\;\!\cdot\;\!)\big)
\in\linf\big(X,\B(\C^{d_\pi})\big)
$
imply that
$$
\big\|E_{\phi,\pi,N}(\;\!\cdot\;\!)\big\|_{\B(\C^{\d_\pi})}^2
\|\varphi(\;\!\cdot\;\!)\|_{\C^{\d_\pi}}^2
\le{\rm Const.}\;\!\|\varphi(\;\!\cdot\;\!)\|_{\C^{\d_\pi}}^2
\in\lone(X,\mu_X).
$$
So, we can apply Lebesgue's dominated convergence theorem in \eqref{eq_Lebesgue} in
conjunction with the equality
$\lim_{N\to\infty}\big\|E_{\phi,\pi,N}(x)\big\|_{\B(\C^{\d_\pi})}^2=0$ for
$\mu_X$-almost every $x\in X$ to get
\begin{align*}
&\lim_{N\to\infty}\left\|\left(D_{\phi,\pi,N}
-i\;\!(\d\pi)_{e_G}\big((P_\phi M_\phi)(\;\!\cdot\;\!)\big)\right)
\sum_{k=1}^{d_\pi}\varphi_k\otimes\pi_{jk}\right\|_{\H^{(\pi)}_j}^2\\
&\le{\rm Const.}\;\!(d_\pi)^{-1}\int_X\d\mu_X(x)\,
\lim_{N\to\infty}\big\|E_{\phi,\pi,N}(x)\big\|_{\B(\C^{\d_\pi})}^2
\|\varphi(x)\|_{\C^{\d_\pi}}^2\\
&=0.
\end{align*}
\end{proof}

Combining the results of Theorem \ref{thm_mixing} and Lemma \ref{lemma_D_phi_pi} we
obtain following criterion for the mixing property of
$U_{\phi,\pi,j}$ in $\H^{(\pi)}_j$ (and thus for the mixing property of the Koopman
operator $U_\phi$ in a subspace of $\H$).

\begin{Theorem}[Mixing property of $U_{\phi,\pi,j}$]\label{thm_mixing_D_pi}
Assume that $\L_Y\phi$ exists $\mu_X$-almost everywhere, that $M_\phi\in\ltwo(X,\g)$,
and that $\L_Y(\pi\circ\phi)\in\linf\big(X,\B(\C^{d_\pi})\big)$. Then, the strong
limit $D_{\phi,\pi}$ exists and is equal to
$i\;\!(\d\pi)_{e_G}\big((P_\phi M_\phi)(\;\!\cdot\;\!)\big)$, and
\begin{enumerate}
\item[(a)]
$\lim_{N\to\infty}\big\langle\varphi,\big(U_{\phi,\pi,j}\big)^N\psi\big\rangle=0$ for
each $\varphi\in\ker(D_{\phi,\pi})^\perp$ and $\psi\in\H^{(\pi)}_j$,
\item[(b)] $U_{\phi,\pi,j}\big|_{\ker(D_{\phi,\pi})^\perp}$ has purely continuous
spectrum.
\end{enumerate}
\end{Theorem}

We are now in position to present a criterion for the presence of a purely absolutely
continuous component in the spectrum of $U_{\phi,\pi,j}$. For this, we recall from
Lemma \ref{Lemma_A_B}(a) that if $D_{\phi,\pi}\in C^1(A)$ then the operator
$$
\mathscr A_{D_{\phi,\pi}}\sum_{k=1}^{d_\pi}\varphi_k\otimes\pi_{jk}
:=\big(AD_{\phi,\pi}+D_{\phi,\pi}A\big)\sum_{k=1}^{d_\pi}\varphi_k\otimes\pi_{jk},
\quad\sum_{k=1}^{d_\pi}\varphi_k\otimes\pi_{jk}\in\dom(A),
$$
is essentially self-adjoint in $\H^{(\pi)}_j$, and its closure
$A_{D_{\phi,\pi}}:=\overline{\mathscr A_{D_{\phi,\pi}}}$ has domain
$$
\dom\big(A_{D_{\phi,\pi}}\big)
=\left\{\sum_{k=1}^{d_\pi}\varphi_k\otimes\pi_{jk}\in\H^{(\pi)}_j
\mid D_{\phi,\pi}\sum_{k=1}^{d_\pi}\varphi_k\otimes\pi_{jk}\in\dom(A)\right\}.
$$
We also introduce the infimum
$$
a_{\phi,\pi}:=\essinf_{x\in X}\inf_{v\in\C^{d_\pi}\!,\,\|v\|_{\C^{d_\pi}}=1}
\left\langle v,\big(i\;\!(\d\pi)_{e_G}\big((P_\phi M_\phi)(x)\big)\big)^2v
\right\rangle_{\C^{d_\pi}},
$$
which is well-defined under the assumptions of Lemma \ref{lemma_dif_pi} since
$i\;\!(\d\pi)_{e_G}\big((P_\phi M_\phi)(x)\big)\in\B(\C^{d_\pi})$ exists and is
Hermitian for $\mu_X$-almost every $x\in X$. Moreover, the value of $a_{\phi,\pi}$ is
invariant under the relation of $C^1$-cohomology. Indeed, if
$\phi,\zeta,\delta\in C^1(X,G)$ are such that
$$
\phi(x)=\zeta(x)^{-1}\;\!\delta(x)\;\!(\zeta\circ F_1)(x)
\quad\hbox{for each $x\in X$,}
$$
then Lemma \ref{lemma_inv_degrees}(d) implies that
\begin{align}
a_{\delta,\pi}
&=\essinf_{x\in X}\inf_{v\in\C^{d_\pi}\!,\,\|v\|_{\C^{d_\pi}}=1}
\left\langle v,\big(i\;\!(\d\pi)_{e_G}\big((P_\delta M_\delta)(x)\big)\big)^2v
\right\rangle_{\C^{d_\pi}}\nonumber\\
&=\essinf_{x\in X}\inf_{v\in\C^{d_\pi}\!,\,\|v\|_{\C^{d_\pi}}=1}
\left\langle v,\big(i\;\!\Ad_{(\pi\circ\zeta)(x)}(\d\pi)_{e_G}
\big((P_\phi M_\phi)(x)\big)\big)^2v\right\rangle_{\C^{d_\pi}}\nonumber\\
&=\essinf_{x\in X}\inf_{v\in\C^{d_\pi}\!,\,\|v\|_{\C^{d_\pi}}=1}
\left\langle v,\Ad_{(\pi\circ\zeta)(x)}\big(i\;\!(\d\pi)_{e_G}
\big((P_\phi M_\phi)(x)\big)\big)^2v\right\rangle_{\C^{d_\pi}}\nonumber\\
&=\essinf_{x\in X}\inf_{(\pi\circ\zeta)(x)u\in\C^{d_\pi}\!,\,
\|(\pi\circ\zeta)(x)u\|_{\C^{d_\pi}}=1}\left\langle u,\big(i\;\!(\d\pi)_{e_G}
\big((P_\phi M_\phi)(x)\big)\big)^2u\right\rangle_{\C^{d_\pi}}\nonumber\\
&=\essinf_{x\in X}\inf_{u\in\C^{d_\pi}\!,\,\|u\|_{\C^{d_\pi}}=1}
\left\langle u,\big(i\;\!(\d\pi)_{e_G}\big((P_\phi M_\phi)(x)\big)\big)^2u
\right\rangle_{\C^{d_\pi}}\nonumber\\
&=a_{\phi,\pi}.\label{eq_a_delta_pi}
\end{align}

\begin{Theorem}[Absolutely continuous spectrum of $U_{\phi,\pi,j}$]\label{thm_absolute_D_pi}
Assume that
\begin{enumerate}
\item[(i)] $\L_Y\phi$ exists $\mu_X$-almost everywhere, $M_\phi\in\ltwo(X,\g)$, and
$\L_Y(\pi\circ\phi)\in\linf\big(X,\B(\C^{d_\pi})\big)$,
\item[(ii)] $\lim_{N\to\infty}\left\|\frac1N\sum_{n=0}^{N-1}W_{\pi\circ\phi}^n
M_{\pi\circ\phi}-(\d\pi)_{e_G}\big((P_\phi M_\phi)(\;\!\cdot\;\!)\big)\right\|
_{\linf(X,\B(\C^{\d_\pi}))}=0$,
\item[(iii)] $\L_Y\big((\d\pi)_{e_G}\big((P_\phi M_\phi)(\;\!\cdot\;\!)\big)\big)
\in\linf\big(X,\B(\C^{d_\pi})\big)$,
\item[(iv)] $M_{\pi\circ\phi}\in C^{+0}\big(A_{D_{\phi,\pi}}\big)$,
\item[(v)] $a_{\phi,\pi}>0$.
\end{enumerate}
Then, $U_{\phi,\pi,j}$ has purely absolutely continuous spectrum.
\end{Theorem}

\begin{proof}
The proof consists in verifying the assumptions of Theorem \ref{thm_absolute}. For
this, we first note from Lemma \ref{lemma_A} that $U_{\phi,\pi,j}\in C^1(A)$ with
$[A,U_{\phi,\pi,j}]=iM_{\pi\circ\phi}\;\!U_{\phi,\pi,j}$. We also note from Lemma
\ref{lemma_D_phi_pi} that the strong limit $D_{\phi,\pi}$ exists and is equal to
$i\;\!(\d\pi)_{e_G}\big((P_\phi M_\phi)(\;\!\cdot\;\!)\big)$. Therefore, the condition
(ii) and the calculation \eqref{eq_Lebesgue} imply that the uniform limit
$$
D_{\phi,\pi}
=\ulim_{N\to\infty}\frac1N\big[A,\big(U_{\phi,\pi,j}\big)^N\big]
\big(U_{\phi,\pi,j}\big)^{-N}
=\ulim_{N\to\infty}\frac iN\sum_{n=0}^{N-1}W_{\pi\circ\phi}^nM_{\pi\circ\phi}
$$
exists. In order to verify the assumption $D_{\phi,\pi}\in C^1(A)$ of Theorem
\ref{thm_absolute}, we take functions $\varphi_k\in C^1(X)$ and observe that
\begin{align*}
&\left\langle A\sum_{k=1}^{d_\pi}\varphi_k\otimes\pi_{jk},
D_{\phi,\pi}\sum_{k'=1}^{d_\pi}\varphi_{k'}\otimes\pi_{jk'}
\right\rangle_{\H^{(\pi)}_j}
-\left\langle\sum_{k=1}^{d_\pi}\varphi_k\otimes\pi_{jk},
D_{\phi,\pi}A\sum_{k'=1}^{d_\pi}\varphi_{k'}\otimes\pi_{jk'}
\right\rangle_{\H^{(\pi)}_j}\\
&=\left\langle\sum_{k=1}^{d_\pi}\varphi_k\otimes\pi_{jk},
-\L_Y\big((\d\pi)_{e_G}\big((P_\phi M_\phi)(\;\!\cdot\;\!)\big)\big)
\sum_{k'=1}^{d_\pi}\varphi_{k'}\otimes\pi_{jk'}\right\rangle_{\H^{(\pi)}_j},
\end{align*}
with $\L_Y\big((\d\pi)_{e_G}\big((P_\phi M_\phi)(\;\!\cdot\;\!)\big)\big)$ the
matrix-valued multiplication operator in $\H^{(\pi)}_j$ given by
$$
\L_Y\big((\d\pi)_{e_G}\big((P_\phi M_\phi)(\;\!\cdot\;\!)\big)\big)
\sum_{k=1}^{d_\pi}\varphi_k\otimes\pi_{jk}
:=\sum_{k,\ell=1}^{d_\pi}\L_Y\big((\d\pi)_{e_G}
\big((P_\phi M_\phi)(\;\!\cdot\;\!)\big)\big)_{k\ell}\;\!\varphi_\ell\otimes\pi_{jk}.
$$
Since the operator
$\L_Y\big((\d\pi)_{e_G}\big((P_\phi M_\phi)(\;\!\cdot\;\!)\big)\big)$ is bounded in
$\H^{(\pi)}_j$ due to condition (iii) and since the set of elements
$\sum_{k=1}^{d_\pi}\varphi_k\otimes\pi_{jk}$ with $\varphi_k\in C^1(X)$ is dense in
$\H^{(\pi)}_j$, this implies that $D_{\phi,\pi}\in C^1(A)$. In order to verify the
assumption $[A,U_{\phi,\pi,j}]\in C^{+0}\big(A_{D_{\phi,\pi}}\big)$ of Theorem
\ref{thm_absolute} we note that since $U_{\phi,\pi,j}\in C^1(A)$ and
$[D_\pi,U_{\phi,\pi,j}]=0$, Lemma \ref{Lemma_A_B}(b) implies that
$U_{\phi,\pi,j}\in C^1\big(A_{D_{\phi,\pi}}\big)$. This, together with condition
(iv), implies that
$
[A,U_{\phi,\pi,j}]=iM_{\pi\circ\phi}U_{\phi,\pi,j}
\in C^{+0}\big(A_{D_{\phi,\pi}}\big)
$
(see \cite[Prop.~5.2.3(b)]{ABG96}). To conclude, we note that a direct calculation and
condition (v) give
$$
\big(D_{\phi,\pi}\big)^2
\ge\essinf_{x\in X}\inf_{v\in\C^{d_\pi}\!,\,\|v\|_{\C^{d_\pi}}=1}
\left\langle v,\big(i\;\!(\d\pi)_{e_G}\big((P_\phi M_\phi)(x)\big)\big)^2v
\right\rangle_{\C^{d_\pi}}
=a_{\phi,\pi}
>0.
$$
This proves the last assumption of Theorem \ref{thm_absolute} for the set
$\Theta=\S^1$.
\end{proof}

\begin{Remark}[Degree of the cocycle $\phi$]\label{remark_degree}
The function $P_\phi M_\phi:X\to\g$ can be interpreted as a matrix-valued degree of
the cocycle $\phi:X\to G$. It generalises the notions of degrees of a cocycle
appearing in \cite{Anz51,Fra00,Fur61,GLL91,Iwa97_2,ILR93} when $G$ is a torus, in
\cite{Fra00_2,Fra04} when $X$ is a torus and $G=\SU(2)$, and in \cite[Sec.~3.3]{Kar13}
when $X$ is a torus and $G$ is semisimple. (In \cite{Fra00_2,Fra04}, K. Fr{\c{a}}czek
calls degree the norm $\|(P_\phi M_\phi)(\;\!\cdot\;\!)\|_\g$ instead $P_\phi M_\phi$
itself, but in the case $G=\SU(2)$ and $F_1$ an irrational rotation that he considers
the difference is not relevant. For more complex/higher-dimensional Lie groups $G$ and
more general maps $F_1$, the distinction becomes relevant. In \cite[Sec.~5.2]{Kar13},
N. Karaliolios calls energy the norm $\|(P_\phi M_\phi)(\;\!\cdot\;\!)\|_\g$ to
emphasize the distinction.) In an analogous way, the function
$P_{\pi\circ\phi}M_{\pi\circ\phi}:X\to\g_\pi$ can be interpreted as a matrix-valued
degree of the cocycle $\pi\circ\phi:X\to\pi(G)$, and Remark \ref{rem_P_pi_phi} shows
that the degree $P_{\pi\circ\phi}M_{\pi\circ\phi}$ of $\pi\circ\phi$ is equal to
$(\d\pi)_{e_G}\big((P_\phi M_\phi)(\;\!\cdot\;\!)\big)$, the image of $P_\phi M_\phi$
under the differential (pushforward) $(\d\pi)_{e_G}:\g\to\g_\pi$. Both degrees
$P_\phi M_\phi$ and $P_{\pi\circ\phi}M_{\pi\circ\phi}$ transform in a natural way
under Lie group homomorphisms and under the relation of $C^1$-cohomology (see Lemma
\ref{lemma_inv_degrees}).

With this in mind, the result of Theorem \ref{thm_mixing_D_pi} can be rephrased
informally as follows: if some regularity assumptions on $\phi$ are satisfied (for
instance if $\phi\in C^1(X,G)$), then the degree $P_{\pi\circ\phi}M_{\pi\circ\phi}$
exists and the operator $U_\phi$ is mixing in the orthocomplement of the kernel of
$P_{\pi\circ\phi}M_{\pi\circ\phi}$ in $\H^{(\pi)}_j$. Similarly, the result of Theorem
\ref{thm_absolute_D_pi} can be rephrased as follows: if some additional assumptions
are satisfied, then the operator $U_\phi$ has purely absolutely continuous spectrum in
$\H^{(\pi)}_j$ if $(iP_{\pi\circ\phi}M_{\pi\circ\phi})^2$ is strictly positive.
Summing up these individual results for the operators $U_{\phi,\pi,j}$ in the
subspaces $\H^{(\pi)}_j$, one obtains a global result for the mixing property and the
absolutely continuous spectrum of the Koopman operator $U_\phi$ in the whole Hilbert
space $\H$.
\end{Remark}

\begin{Remark}
Theorem \ref{thm_absolute_D_pi} generalises in various ways Theorem 3.5 of
\cite{Tie15_2}. First, the base space $X$ here is only supposed to be compact manifold
and not a compact Lie group as in \cite{Tie15_2}. Furthermore, the use of the operator
$D_{\phi,\pi}$ in the definition of the conjugate operator $A_{D_{\phi,\pi},N}$ (see
Proposition \ref{prop_conjugate}) allowed us to remove a commutation assumption made
in \cite[Thm.~3.5]{Tie15_2} which forced the matrix-valued function $\pi\circ\phi$ to
be diagonal up to conjugation (in applications, this forced the cocycle $\phi$ to be
trivially cohomologous to a diagonal cocycle). Here, this commutation assumption is to
some extent replaced by the commutation relation
$[D_{\phi,\pi},(U_{\phi,\pi,j})^n]=0$ which is automatically satisfied for each
$n\in\Z$. Finally, the positivity assumption $a_{\phi,\pi}>0$ in Theorem
\ref{thm_absolute_D_pi} given in terms of the function
$(\d\pi)_{e_G}\big((P_\phi M_\phi)(\;\!\cdot\;\!)\big)$ is more natural and less
restrictive than the positivity assumption in \cite[Thm.~3.5]{Tie15_2} given in terms
of the sequence $\frac1N\sum_{n=0}^{N-1}M_{\pi\circ\phi}\circ F_n$. Indeed, the
assumption $a_{\phi,\pi}>0$ is more natural because it does not change under the
relation of  $C^1$-cohomology due to \eqref{eq_a_delta_pi}, and it is less restrictive
because in general it is the sequence
$\frac1N\sum_{n=0}^{N-1}W_{\pi\circ\phi}^nM_{\pi\circ\phi}$ and its limit
$(\d\pi)_{e_G}\big((P_\phi M_\phi)(\;\!\cdot\;\!)\big)$ which encode the asymptotic
behaviour of the operator $U_{\phi,\pi,j}$ and not the sequence
$\frac1N\sum_{n=0}^{N-1}M_{\pi\circ\phi}\circ F_n$ (Lemma \ref{lemma_ergodic} below
furnish an illustration of this fact).
\end{Remark}

We now present two particular cases where the condition (ii) of Theorem
\ref{thm_absolute_D_pi} is satisfied and the function $P_\phi M_\phi$ takes a simple
form. For this, we need to fix some notations. We write
$$
\int_X\d\mu_X(x)\,M_\phi(x)
\quad\hbox{and}\quad
\int_{X\times G}\d(\mu_X\otimes\mu_G)(x,g)\,\Ad_gM_\phi(x)
$$
for the the strong integrals of the functions $X\ni x\mapsto M_\phi(x)\in\g$ and
$X\times G\ni(x,g)\mapsto\Ad_gM_\phi(x)\in\g$ (if they exist). We write
$$
\g^{\Ad}:=\big\{Z\in\g\mid\Ad_gZ=Z~\hbox{for all $g\in G$}\big\}
$$
for the invariant subspace of the adjoint representation $\Ad$ of $G$. Finally, we let
$P_{\Ad}\in\B(\g)$ be the operator given by
\begin{equation}\label{eq_P_Ad}
P_{\Ad}(Z):=\int_G\d\mu_G(g)\Ad_gZ,\quad Z\in\g,
\end{equation}
and we note that since $\Ad$ is a finite-dimensional unitary representation of $G$ on
$\g$, $P_{\Ad}$ coincides with the orthogonal projection onto $\g^{\Ad}$ (see
\cite[Lemma~12.16]{Hal15}).

\begin{Lemma}[$F_1$ uniquely ergodic and $\pi\circ\phi$ diagonal]\label{lemma_diagonal}
Assume that $\phi\in C^1(X,G)$, that $F_1$ is uniquely ergodic, and that $\pi\circ\phi$
is diagonal (that is, $\pi_{k\ell}\circ\phi=(\pi_{k\ell}\circ\phi)\;\!\delta_{k\ell}$
for each $k,\ell\in\{1,\ldots,d_\pi\}$). Then,
$$
\lim_{N\to\infty}\left\|\frac1N\sum_{n=0}^{N-1}W_{\pi\circ\phi}^nM_{\pi\circ\phi}
-(\d\pi)_{e_G}(M_{\phi,\star})\right\|_{\linf(X,\B(\C^{\d_\pi}))}=0,
$$
with
$$
M_{\phi,\star}:=\int_X\d\mu_X(x)\,M_\phi(x).
$$
Furthermore, if $\pi\circ\phi$ is diagonal for each $\pi\in\widehat G$, then
$P_\phi M_\phi=M_{\phi,\star}$ $\mu_X$-almost everywhere.
\end{Lemma}

\begin{proof}
The diagonality of $\pi\circ\phi$  implies that $M_{\pi\circ\phi}$ is diagonal and that
$W_{\pi\circ\phi}^nM_{\pi\circ\phi}=M_{\pi\circ\phi}\circ F_n$, and the condition
$\phi\in C^1(X,G)$ implies that $M_\phi\in C(X,\g)$ and $M_{\pi\circ\phi}\in C(X,\g_\pi)$.
So, the strong integral
$$
M_{\phi,\star}=\int_X\d\mu_X(x)\,M_\phi(x)\in\g
$$
exists. Therefore, it follows from the linearity and continuity of $(\d\pi)_{e_G}$,
from \eqref{eq_d_M_phi}, and from the unique ergodicity of $F_1$ that
\begin{align}
&\lim_{N\to\infty}\left\|\frac1N\sum_{n=0}^{N-1}W_{\pi\circ\phi}^nM_{\pi\circ\phi}
-(\d\pi)_{e_G}(M_{\phi,\star})\right\|_{\linf(X,\B(\C^{\d_\pi}))}\nonumber\\
&=\lim_{N\to\infty}\left\|\frac1N\sum_{n=0}^{N-1}M_{\pi\circ\phi}\circ F_n
-\int_X\d\mu_X(x)\,(\d\pi)_{e_G}\big(M_\phi(x)\big)\right\|
_{\linf(X,\B(\C^{\d_\pi}))}\nonumber\\
&=\lim_{N\to\infty}\left\|\frac1N\sum_{n=0}^{N-1}M_{\pi\circ\phi}\circ F_n
-\int_X\d\mu_X(x)\,M_{\pi\circ\phi}(x)\right\|
_{\linf(X,\B(\C^{\d_\pi}))}\nonumber\\
&\le{\rm Const.}\lim_{N\to\infty}\;\!\max_{k\in\{1,\ldots,d_\pi\}}
\left\|\frac1N\sum_{n=0}^{N-1}\big(M_{\pi\circ\phi}\big)_{kk}\circ F_n
-\int_X\d\mu_X(x)\,\big(M_{\pi\circ\phi}\big)_{kk}(x)\right\|_{\linf(X,\mu_X)}\nonumber\\
&=0.\label{eq_d_pi_M_star}
\end{align}

For the second claim, we note that \eqref{eq_d_pi_PM} and \eqref{eq_d_pi_M_star} imply
for $\mu_X$-almost every $x\in X$ that
$$
0=(\d\pi)_{e_G}\big((P_\phi M_\phi)(x)\big)-(\d\pi)_{e_G}(M_{\phi,\star})
=(\d\pi)_{e_G}\big(\big(P_\phi M_\phi\big)(x)-M_{\phi,\star}\big).
$$
Thus, $\big(P_\phi M_\phi\big)(x)-M_{\phi,\star}\in\ker(\d\pi)_{e_G}$ for
$\mu_X$-almost every $x\in X$. Since this holds for each $\pi\in\widehat G$, we even
get that
$$
\big(P_\phi M_\phi\big)(x)-M_{\phi,\star}
\in\bigcap_{\pi\in\widehat G}\ker(\d\pi)_{e_G}
$$
for $\mu_X$-almost every $x\in X$. Therefore, it is sufficient to show that
$\bigcap_{\pi\in\widehat G}\ker(\d\pi)_{e_G}=\{0\}$ to conclude. So, suppose by absurd
that there exists $Z\in\bigcap_{\pi\in\widehat G}\ker(\d\pi)_{e_G}\setminus\{0\}$, and
take the Lie group isomorphism $\pi_*:G\to\pi_*(G)\subset\U(n)$ mentioned before Lemma 
\ref{lemma_W}. Since $Z\ne0$ and $(\d\pi_*)_{e_G}:\g\to\g_{\pi_*}$ is a Lie algebra
isomorphism \cite[Rem.~2.1.52(i)]{BLU07}, we have $(\d\pi_*)_{e_G}(Z)\ne0$. On another
hand, since $\pi_*$ is a finite-dimensional unitary representation of $G$, there exist
$m\in\N^*$ and $\pi_1,\ldots,\pi_m\in\widehat G$ such that
$\pi_*=\pi_1\oplus\cdots\oplus\pi_m$. Therefore,
$$
(\d\pi_*)_{e_G}(Z)
=\big(\d(\pi_1\oplus\cdots\oplus\pi_m)\big)_{e_G}(Z)
=\big((\d\pi_1)_{e_G}(Z),\ldots,(\d\pi_m)_{e_G}(Z)\big)
=(0,\ldots,0)
=0
$$
due to the inclusion $Z\in\bigcap_{\pi\in\widehat G}\ker(\d\pi)_{e_G}$. This is a
contradiction. Thus, $\bigcap_{\pi\in\widehat G}\ker(\d\pi)_{e_G}=\{0\}$, and the
claim is proved.
\end{proof}

\begin{Lemma}[$T_\phi$ uniquely ergodic]\label{lemma_ergodic}
Assume that $\phi\in C^1(X,G)$ and that $T_\phi$ is uniquely ergodic. Then,
$$
\lim_{N\to\infty}\left\|\frac1N\sum_{n=0}^{N-1}W_{\pi\circ\phi}^nM_{\pi\circ\phi}
-(\d\pi)_{e_G}(M_{\phi,\star})\right\|_{\linf(X,\B(\C^{\d_\pi}))}=0,
$$
with
\begin{equation}\label{eq_M_phi_star_Ad}
M_{\phi,\star}
:=\int_{X\times G}\d(\mu_X\otimes\mu_G)(x,g)\,\Ad_gM_\phi(x)
=P_{\Ad}\left(\int_X\d\mu_X(x)\,M_\phi(x)\right).
\end{equation}
Furthermore, $P_\phi M_\phi=M_{\phi,\star}$ $\mu_X$-almost everywhere.
\end{Lemma}

\begin{proof}
The condition $\phi\in C^1(X,G)$ implies that $M_\phi\in C(X,\g)$ and
$M_{\pi\circ\phi}\in C(X,\g_\pi)$. So, the strong integral
$$
M_{\phi,\star}=\int_{X\times G}\d(\mu_X\otimes\mu_G)(x,g)\,\Ad_gM_\phi(x)\in\g
$$
exists. Define for $v\in\C^{d_\pi}$ the function
$$
f_v:X\times G\to\C,~~(x,g)\mapsto
\big\langle v,\Ad_{\pi(g)}M_{\pi\circ\phi}(x)v\big\rangle_{\C^{d_\pi}}.
$$
We have $f_v\in C(X\times G)$, and \eqref{eq_d_M_phi} implies that
$$
f_v(x,g)=\big\langle v,(\d\pi)_{e_G}
\big(\Ad_gM_\phi(x)\big)v\big\rangle_{\C^{d_\pi}}
\quad\hbox{for each $(x,g)\in X\times G$.}
$$
So, it follows from the linearity and continuity of $(\d\pi)_{e_G}$ that
\begin{align*}
\int_{X\times G}\d(\mu_X\otimes\mu_G)(x,g)\,f_v(x,g)
&=\int_{X\times G}\d(\mu_X\otimes\mu_G)(x,g)\,
\big\langle v,(\d\pi)_{e_G}\big(\Ad_gM_\phi(x)\big)v\big\rangle_{\C^{d_\pi}}\\
&=\left\langle v,(\d\pi)_{e_G}\left(\int_{X\times G}
\d(\mu_X\otimes\mu_G)(x,g)\,\Ad_gM_\phi(x)\right)v\right\rangle_{\C^{d_\pi}}\\
&=\big\langle v,(\d\pi)_{e_G}(M_{\phi,\star})v\big\rangle_{\C^{d_\pi}}.
\end{align*}
On another hand, a calculation as in \eqref{eq_W_and_tilde} shows that
$$
\frac1N\sum_{n=0}^{N-1}\big(f_v\circ T_\phi^n\big)(x,g)
=\left\langle v,\frac1N\sum_{n=0}^{N-1}\Ad_{\pi(g)}
\big(W_{\pi\circ\phi}^nM_{\pi\circ\phi}\big)(x)v\right\rangle_{\C^{d_\pi}}
\quad\hbox{for each $(x,g)\in X\times G$.}
$$
Therefore, we infer from the unique ergodicity of $T_\phi$ that for each
$(x,g)\in X\times G$
\begin{align*}
&\lim_{N\to\infty}\left\langle v,\left(\frac1N\sum_{n=0}^{N-1}\Ad_{\pi(g)}
\big(W_{\pi\circ\phi}^nM_{\pi\circ\phi}\big)(x)
-(\d\pi)_{e_G}(M_{\phi,\star})\right)v\right\rangle_{\C^{d_\pi}}\\
&=\lim_{N\to\infty}\left(\frac1N\sum_{n=0}^{N-1}\big(f_v\circ T_\phi^n\big)(x,g)
-\int_{X\times G}\d(\mu_X\otimes\mu_G)(x,g)\,f_v(x,g)\right)\\
&=0.
\end{align*}
Setting $g=e_G$, we thus get for each $x\in X$
\begin{equation}\label{eq_scalar}
\lim_{N\to\infty}\left\langle v,\left(\frac1N\sum_{n=0}^{N-1}
\big(W_{\pi\circ\phi}^nM_{\pi\circ\phi}\big)(x)
-(\d\pi)_{e_G}(M_{\phi,\star})\right)v\right\rangle_{\C^{d_\pi}}=0.
\end{equation}
Now, the extreme value theorem for continuous functions on compact sets implies the
existence of $x_0\in X$ and $v_0\in\C^{d_\pi}$ with $\|v_0\|_{\C^{d_\pi}}=1$ such that
\begin{align*}
&\left\|\frac1N\sum_{n=0}^{N-1}W_{\pi\circ\phi}^nM_{\pi\circ\phi}
-(\d\pi)_{e_G}(M_{\phi,\star})\right\|_{\linf(X,\B(\C^{\d_\pi}))}\\
&=\sup_{x\in X}\;\!\sup_{v\in\C^{d_\pi}\!,\,\|v\|_{\C^{d_\pi}}=1}
\left|\left\langle v,\left(\frac1N\sum_{n=0}^{N-1}\big(W_{\pi\circ\phi}^n
M_{\pi\circ\phi}\big)(x)-(\d\pi)_{e_G}(M_{\phi,\star})\right)v
\right\rangle_{\C^{d_\pi}}\right|\\
&=\left|\left\langle v_0,\left(\frac1N\sum_{n=0}^{N-1}\big(W_{\pi\circ\phi}^n
M_{\pi\circ\phi}\big)(x_0)-(\d\pi)_{e_G}(M_{\phi,\star})\right)v_0
\right\rangle_{\C^{d_\pi}}\right|.
\end{align*}
Thus, we deduce from \eqref{eq_scalar} that
\begin{align*}
&\lim_{N\to\infty}\left\|\frac1N\sum_{n=0}^{N-1}W_{\pi\circ\phi}^nM_{\pi\circ\phi}
-(\d\pi)_{e_G}(M_{\phi,\star})\right\|_{\linf(X,\B(\C^{\d_\pi}))}\\
&=\lim_{N\to\infty}\left|\left\langle v_0,\left(\frac1N\sum_{n=0}^{N-1}
\big(W_{\pi\circ\phi}^nM_{\pi\circ\phi}\big)(x_0)
-(\d\pi)_{e_G}(M_{\phi,\star})\right)v_0\right\rangle_{\C^{d_\pi}}\right|\\
&=0,
\end{align*}
which proves the first equality in \eqref{eq_M_phi_star_Ad}.

The second equality in \eqref{eq_M_phi_star_Ad} follows directly from the definition 
\eqref{eq_P_Ad} of the operator $P_{\Ad}$, and the fact that
$P_\phi M_\phi=M_{\phi,\star}$ $\mu_X$-almost everywhere can be shown as in the proof
of Lemma \ref{lemma_diagonal}.
\end{proof}

Lemma \ref{lemma_ergodic} furnishes, as a by-product, criteria for the non (unique)
ergodicity of skew products $T_\phi$ which will be useful in applications. To state
them, we need some additional notations. We denote by
$[\;\!\cdot\;\!,\;\!\cdot\;\!]_\g:\g\times\g\to\g$ the Lie bracket of $\g$, we write
$$
z(\g):=\big\{Z\in\g\mid[X,Z]_\g=0~\hbox{for all $X\in\g$}\big\}
$$
for the center of the Lie algebra $\g$, and we write $\big(\g^{\Ad}\big)^\perp$ for
the orthocomplement of $\g^{\Ad}$ in $\g$.

\begin{Theorem}[Non ergodicity of $T_\phi$]\label{theorem_non_ergodic}
Assume that $\phi\in C^1(X,G)$ and that $\phi$ has nonzero degree $P_\phi M_\phi$,
that is, $P_\phi M_\phi\ne0$ on a set $X_0\subset X$ with $\mu_X(X_0)>0$.
\begin{enumerate}
\item[(a)] If $\int_X\d\mu_X(x)\,M_\phi(x)\in\big(\g^{\Ad}\big)^\perp$, then $T_\phi$
is not uniquely ergodic.
\item[(b)] If $G$ is connected and $z(\g)=\{0\}$, then $T_\phi$ is not uniquely
ergodic.
\item[(c)] If the assumptions of (a) or (b) are satisfied and if $F_1$ is uniquely
ergodic, then $T_\phi$ is not ergodic.
\end{enumerate}
\end{Theorem}

\begin{proof}
(a) Suppose by absurd that $T_\phi$ is uniquely ergodic. Then, the inclusion
$\int_X\d\mu_X(x)\,M_\phi(x)\in\big(\g^{\Ad}\big)^\perp$ and \eqref{eq_M_phi_star_Ad}
imply for $\mu_X$-almost every $x\in X_0$ that
$$
0
\ne\big(P_\phi M_\phi\big)(x)
=M_{\phi,\star}
=P_{\Ad}\left(\int_X\d\mu_X(x)\,M_\phi(x)\right)
=0,
$$
which is a contradiction. Therefore, $T_\phi$ is not uniquely ergodic.

(b) If $G$ is connected, then the exponential map $\exp:\g\to G$ is surjective and
$\g^{\Ad}=z(\g)$ \cite[Eq.~3.1.13]{DK00}. Since $z(\g)=\{0\}$, this implies that
$\big(\g^{\Ad}\big)^\perp=\g$. Thus
$\int_X\d\mu_X(x)\,M_\phi(x)\in\big(\g^{\Ad}\big)^\perp$, and the claim follows from
point (a).

(c) The claim follows from the fact that unique ergodicity and ergodicity are
equivalent properties for $T_\phi$ if $F_1$ is uniquely ergodic (see
\cite[Thm.~4.21]{EW11}).
\end{proof}

\begin{Remark}\label{remark_semisimple}
(a) If $G$ is connected, $\phi\in C^1(X,G)$, and $T_\phi$ is uniquely ergodic, then
$\g^{\Ad}=z(\g)$ and \eqref{eq_M_phi_star_Ad} implies that the degree
$P_\phi M_\phi=M_{\phi,\star}$ belongs to the center $z(\g)$ of the Lie algebra $\g:$
\begin{equation}\label{eq_in_center}
M_{\phi,\star}=P_{\Ad}\left(\int_X\d\mu_X(x)\,M_\phi(x)\right)\in z(\g).
\end{equation}
Furthermore, if $(\d\pi)_{e_G}$ is surjective, then it is known that the differential
$(\d\pi)_{e_G}$ maps $z(\g)$ into the center $z(\g_\pi)$ of $\g_\pi$. So,
\eqref{eq_in_center} implies that $(\d\pi)_{e_G}(M_{\phi,\star})\in z(\g_\pi)$.

(b) If $G$ is a connected and semisimple, then we have $z(\g)=\{0\}$
\cite[Prop.~1.13]{Kna02}. Thus, Theorem \ref{theorem_non_ergodic}(b) implies that
there does not exist uniquely ergodic skew products $T_\phi$ with $\phi\in C^1(X,G)$
and nonzero degree if $G$ is a connected semisimple compact Lie group. This applies
for example in the cases $G=\SU(n)$ and $G=\SO(n+1,\R)$ for $n\ge2$. On the hand, if
the assumptions of Theorem \ref{theorem_non_ergodic}(a)-(b) are not satisfied, then
skew products $T_\phi$ with $\phi\in C^1(X,G)$ and nonzero degree can be uniquely
ergodic. This occurs for example in the case $G$ is a torus (see
\cite[Thm.~2.1]{Fur61}, \cite[Cor.~3]{Iwa97_2}, \cite[p.~8-9]{Tie15_3}, and Section
\ref{section_torus}) or $G=\U(2)$ (see Section \ref{section_U(2)}).
\end{Remark}

Using Lemmas \ref{lemma_diagonal} and \ref{lemma_ergodic}, we obtain the following
corollaries of Theorems \ref{thm_mixing_D_pi} and \ref{thm_absolute_D_pi}.

\begin{Corollary}\label{cor_mixing_D_pi}
Assume that $\phi\in C^1(X,G)$ and suppose either that $F_1$ is uniquely ergodic and
$\pi\circ\phi$ is diagonal, or that $T_\phi$ is uniquely ergodic. Then, the strong
limit $D_{\phi,\pi}$ exists and is equal to $i\;\!(\d\pi)_{e_G}(M_{\phi,\star})$, and
\begin{enumerate}
\item[(a)] $\lim_{N\to\infty}\big\langle\varphi,
\big(U_{\phi,\pi,j}\big)^N\psi\big\rangle=0$ for each
$\varphi\in\ker(D_{\phi,\pi})^\perp$ and $\psi\in\H^{(\pi)}_j$,
\item[(b)] $U_{\phi,\pi,j}\big|_{\ker(D_{\phi,\pi})^\perp}$ has purely continuous
spectrum.
\end{enumerate}
\end{Corollary}

\begin{proof}
Direct consequence of Theorem \ref{thm_mixing_D_pi} and Lemmas \ref{lemma_diagonal}
and \ref{lemma_ergodic}.
\end{proof}

\begin{Corollary}\label{cor_absolute_D_pi}
Assume that
\begin{enumerate}
\item[(i)] $\phi\in C^1(X,G)$,
\item[(ii)] $F_1$ is uniquely ergodic and $\pi\circ\phi$ is diagonal, or $T_\phi$ is
uniquely ergodic,
\item[(iii)] $M_{\pi\circ\phi}\in C^{+0}\big(A_{D_{\phi,\pi}}\big)$,
\item[(iv)] $\det\big((\d\pi)_{e_G}(M_{\phi,\star})\big)\ne0$.
\end{enumerate}
Then, $U_{\phi,\pi,j}$ has purely absolutely continuous spectrum.
\end{Corollary}

\begin{proof}
The proof consists in verifying the assumptions of Theorem \ref{thm_absolute_D_pi}. In
view of Lemmas \ref{lemma_diagonal} and \ref{lemma_ergodic}, it is clear that
assumptions (i)-(iii) imply assumptions (i)-(iv) of Theorem \ref{thm_absolute_D_pi}.
Furthermore, since
$(\d\pi)_{e_G}\big((P_\phi M_\phi)(x)\big)=(\d\pi)_{e_G}(M_{\phi,\star})$ for
$\mu_X$-almost every $x\in X$, it follows from the min-max theorem that
\begin{align*}
a_{\phi,\pi}
&=\essinf_{x\in X}\inf_{v\in\C^{d_\pi},\,\|v\|_{\C^{d_\pi}}=1}\left\langle v,
\big(i\;\!(\d\pi)_{e_G}\big((P_\phi M_\phi)(x)\big)\big)^2v
\right\rangle_{\C^{d_\pi}}\\
&=\inf_{v\in\C^{d_\pi}\!,\,\|v\|_{\C^{d_\pi}}=1}\left\langle v,
\big(i\;\!(\d\pi)_{e_G}(M_{\phi,\star})\big)^2v\right\rangle_{\C^{d_\pi}}\\
&=\hbox{minimal eigenvalue of $\big(i\;\!(\d\pi)_{e_G}(M_{\phi,\star})\big)^2$.}
\end{align*}
In consequence, assumption (iv) implies that $a_{\phi,\pi}>0$, and thus all
assumptions of Theorem \ref{thm_absolute_D_pi} are satisfied.
\end{proof}

\begin{Remark}\label{remark_kernel}
When $D_{\phi,\pi}$ is equal to $i\;\!(\d\pi)_{e_G}(M_{\phi,\star})$ (as in
Corollaries \ref{cor_mixing_D_pi} and \ref{cor_absolute_D_pi}), one has an explicit
formula for $\ker(D_{\phi,\pi})$. Indeed, let $q_{\phi,\pi}\in\U(d_\pi)$ be the
unitary matrix which diagonalises the hermitian matrix
$i\;\!(\d\pi)_{e_G}(M_{\phi,\star})$, that is,
$$
q_{\phi,\pi}\cdot i\;\!(\d\pi)_{e_G}(M_{\phi,\star})\cdot q_{\phi,\pi}^{-1}
=\diag\big(\lambda_1,\ldots,\lambda_{d_\pi}\big)
$$
with $\lambda_1,\ldots,\lambda_{d_\pi}$ the eigenvalues of
$i\;\!(\d\pi)_{e_G}(M_{\phi,\star})$. Then, the multiplication operator $Q_{\phi,\pi}$
in $\H^{(\pi)}_j$ given by
\begin{equation}\label{eq_Q_phi_pi}
Q_{\phi,\pi}\sum_{k=1}^{d_\pi}\varphi_k\otimes\pi_{jk}
=\sum_{k,\ell=1}^{d_\pi}(q_\pi)_{k\ell}\;\!\varphi_\ell\otimes\pi_{jk},
\quad\varphi_k\in\ltwo(X,\mu_X),
\end{equation}
is unitary, and a direct calculation shows that
$$
\ker\big(Q_{\phi\pi}D_{\phi,\pi}Q_{\phi,\pi}^{-1}\big)
=\bigoplus_{k\in\{1,\ldots,d_\pi\},\,\lambda_k=0}\ltwo(X,\mu_X)\otimes\{\pi_{jk}\}.
$$
Therefore, we obtain
$$
\ker(D_{\phi,\pi})
=Q_{\phi,\pi}^{-1}\ker\big(Q_{\phi\pi}D_{\phi,\pi}Q_{\phi,\pi}^{-1}\big)
=Q_{\phi,\pi}^{-1}\left\{\bigoplus_{k\in\{1,\ldots,d_\pi\},\,\lambda_k=0}
\ltwo(X,\mu_X)\otimes\{\pi_{jk}\}\right\},
$$
which implies
\begin{align*}
\ker(D_{\phi,\pi})^\perp
&=Q_{\phi,\pi}^{-1}\left\{\bigoplus_{k\in\{1,\ldots,d_\pi\},\,\lambda_k=0}
\ltwo(X,\mu_X)\otimes\{\pi_{jk}\}\right\}^\perp\\
&=Q_{\phi,\pi}^{-1}\left\{\bigoplus_{k\in\{1,\ldots,d_\pi\},\,\lambda_k\ne0}
\ltwo(X,\mu_X)\otimes\{\pi_{jk}\}\right\},
\end{align*}
due to the orthogonality relation \eqref{eq_orthogonality}.
\end{Remark}

To conclude the section, we recall a result in the case of translations on tori
$\T^d:=(\S^1)^d$ ($d\in\N^*$) which follows from \cite[Lemma~3.7]{Tie15_2} and the
proof of \cite[Lemma~4]{ILR93}:

\begin{Lemma}\label{Lemma_Lebesgue}
Assume that $X=\T^d$ ($d\in\N^*$), and let $\{F_t\}_{t\in\R}$ be given
by
$$
F_t(x):=\big(x_1\e^{2\pi it\alpha_1},\ldots,x_d\e^{2\pi it\alpha_d}\big),
\quad t\in\R,~x=(x_1,\ldots,x_d)\in\T^d,
$$
for some $\alpha=(\alpha_1,\ldots,\alpha_d)\in\R^d$ with
$\alpha_{k_0}\in\R\setminus\Q$ for some $k_0\in\{1,\ldots,d\}$. Then, if
$U_{\phi,\pi,j}$ has purely absolutely continuous spectrum, $U_{\phi,\pi,j}$ has
Lebesgue spectrum with uniform countable multiplicity.
\end{Lemma}

%--------------------------------------------------------------------------------------
\section{Examples}\label{section_examples}
\setcounter{equation}{0}
%--------------------------------------------------------------------------------------

In this section, we apply the results of Section \ref{section_cocycles} in some
examples. In each example, we present both general results and some more explicit
results in particular cases. As a first application, we start the simple, but
instructive, case where the cocycle $\phi$ takes values in a torus.

%--------------------------------------------------------------------------------------
\subsection{Cocycles with values in a torus}\label{section_torus}
%--------------------------------------------------------------------------------------

Assume that $G=\T^{d'}$ ($d'\in\N^*$). Then, $\g=i\;\!\R^{d'}$, each representation
$\pi^{(q)}\in\widehat{\T^{d'}}$ is a character of $\T^{d'}$ given by
$\pi^{(q)}(y):=y_1^{q_1}\cdots y_{d'}^{q_{d'}}$ for $y=(y_1,\ldots,y_{d'})\in\T^{d'}$
and $q=(q_1,\ldots,q_{d'})\in\Z^{d'}$, and
$$
\H^{(q)}
:=\H^{(\pi^{(q)})}_1
=\ltwo(X,\mu_X)\otimes\big\{\pi^{(q)}\big\}.
$$
If $\phi\in C^1(X,\T^{d'})$, Theorem \ref{thm_mixing_D_pi} implies that the strong
limit
$$
D_{\phi,q}
:=D_{\phi,\pi^{(q)}}
=\slim_{N\to\infty}\frac1N\big[A,\big(U_{\phi,\pi^{(q)},1}\big)^N\big]
\big(U_{\phi,\pi^{(q)},1}\big)^{-N}
$$
exists and is equal to
$i\big(\d\pi^{(q)}\big)_{e'}\big((P_\phi M_\phi)(\;\!\cdot\;\!)\big)$ with
$e'=e_{\;\!\T^{d'}}=(1,\ldots,1)$, and
$\lim_{N\to\infty}\big\langle\varphi,\big(U_{\phi,\pi^{(q)},1}\big)^N\psi\big\rangle=0$
for each $\varphi\in\ker(D_{\phi,q})^\perp$ and $\psi\in\H^{(q)}$. Therefore, $U_\phi$
is mixing in the subspace
\begin{equation}\label{eq_ac_torus_1}
\H_{\rm mix}:=\bigoplus_{q\in\Z^{d'}}\ker(D_{\phi,q})^\perp\subset\H.
\end{equation}

In order to apply Corollaries \ref{cor_mixing_D_pi} and  \ref{cor_absolute_D_pi}, we
make the following additional assumptions: $F_1$ is uniquely ergodic and
\begin{equation}\label{eq_condition}
\int_0^1\frac{\d t}t\,\big\|\L_Y(\pi^{(q)}\circ\phi)\circ F_t
-\L_Y(\pi^{(q)}\circ\phi)\big\|_{\linf(X,\mu_X)}<\infty
\quad\hbox{for each $q\in\Z^{d'}$.}
\end{equation}
Under these assumptions, Corollary \ref{cor_mixing_D_pi} implies that
$$
D_{\phi,q}
=i\big(\d\pi^{(q)}\big)_{e'}(M_{\phi,\star})
=i\int_X\d\mu_X(x)\,M_{\pi^{(q)}\circ\phi}(x)
\in\R
$$
with
$$
M_{\phi,\star}=\int_X\d\mu_X(x)\,M_\phi(x)\in i\;\!\R^{d'}
\quad\hbox{and}\quad
M_{\pi^{(q)}\circ\phi}=\L_Y(\pi^{(q)}\circ\phi)(\pi^{(q)}\circ\phi)^{-1}.
$$
In particular, we have
$$
A_{D_{\phi,q}}=AD_{\phi,q}+D_{\phi,q}A=2AD_{\phi,q},
$$
which implies that $C^{+0}(A)\subset C^{+0}\big(A_{D_{\phi,q}}\big)$. So, in order to
verify the assumption (iii) of Corollary \ref{cor_absolute_D_pi} it is sufficient to
check that $M_{\pi^{(q)}\circ\phi}\in C^{+0}(A)$. But, since $\phi\in C^1(X,\T^{d'})$,
we have $(\pi^{(q)}\circ\phi)^{-1}\in C^1(A)$. So, it is sufficient to check that
$\L_Y(\pi^{(q)}\circ\phi)\in C^{+0}(A)$, which follows directly from
\eqref{eq_condition}. Finally, since
$D_{\phi,q}=i\big(\d\pi^{(q)}\big)_{e'}(M_{\phi,\star})$ is scalar, the assumption (iv)
of Corollary \ref{cor_absolute_D_pi} is equivalent to $D_{\phi,q}\ne0$. So, Corollary
\ref{cor_absolute_D_pi} implies that $U_{\phi,\pi^{(q)},1}$ has purely absolutely
continuous spectrum if $D_{\phi,q}\ne0$, and thus that $U_\phi$ has purely absolutely
continuous spectrum in the subspace
\begin{equation}\label{eq_ac_torus_2}
\H_{\rm ac}
:=\bigoplus_{q\in\Z^{d'}\!,\,D_{\phi,q}\ne0}\H^{(q)}
\subset\H_{\rm mix}.
\end{equation}

These results are new in this generality. They extend similar results in the
particular case $X\times G=\T^d\times\T^{d'}$, $d,d'\in\N^*$ (see for instance 
\cite{Anz51,Cho87,Fra00,Iwa97_1,Iwa97_2,ILR93,Kus74,Tie15_2,Tie15_3}). In that case,
if the flow $\{F_t\}_{t\in\R}$ on $X=\T^d$ is given by
$$
F_t(x):=\big(x_1\e^{2\pi it\alpha_1},\ldots,x_d\e^{2\pi it\alpha_d}\big),
\quad t\in\R,~x=(x_1,\ldots,x_d)\in\T^d,
$$
for some $\alpha=\big(\alpha_1,\ldots,\alpha_d\big)\in\R^d$ with
$\alpha_1,\ldots,\alpha_d,1$ rationally independent, one even obtains Lebesgue
spectrum with uniform countable multiplicity due to Lemma \ref{Lemma_Lebesgue}.

%--------------------------------------------------------------------------------------
\subsection{Cocycles with values in $\SU(2)$}\label{section_SU(2)}
%--------------------------------------------------------------------------------------

Assume that
\begin{gather*}
G
=\SU(2)
=\left\{
\begin{pmatrix}
z_1 & z_2\\
-\overline{z_2} & \overline{z_1}
\end{pmatrix}
\mid z_1,z_2\in\C,~|z_1|^2+|z_2|^2=1\right\},\\
\g=\su(2)
=\left\{
\begin{pmatrix}
is & z\\
-\overline z & -is
\end{pmatrix}
\mid s\in\R,~z\in\C\right\},\\
\langle Z_1,Z_2\rangle_{\su(2)}:=\frac12\Tr\big(Z_1Z_2^*\big),\quad Z_1,Z_2\in\su(2).
\end{gather*}
The set $\widehat{\SU(2)}$ of finite-dimensional IUR's of $\SU(2)$ can be described as
follows. For each $\ell\in\N$, let $V_\ell$ be the $(\ell+1)$-dimensional vector space
of homogeneous polinomials of degree $\ell$ in the variables $\omega_1,\omega_2\in\C$.
Endow $V_\ell$ with the basis
$$
p_j(\omega_1,\omega_2):=\omega_1^j\omega_2^{\ell-j},\quad j\in\{0,\ldots,\ell\},
$$
and the scalar product
$\langle\;\!\cdot\;\!,\;\!\cdot\;\!\rangle_{V_\ell}:V_\ell\times V_\ell\to\C$ given by
$$
\left\langle\sum_{j=0}^\ell\alpha_jp_j,\sum_{k=0}^\ell\beta_kp_k\right\rangle_{V_\ell}
:=\sum_{j=0}^\ell j!(q-j)!\;\!\alpha_j\;\!\overline{\beta_j},
\quad\alpha_j,\beta_k\in\C.
$$
Then, the function $\pi^{(\ell)}:\SU(2)\to\U(V_\ell)\simeq\U(\ell+1)$ given by
\begin{equation}\label{eq_rep_SU(2)}
\big(\pi^{(\ell)}(g)p\big)(\omega_1,\omega_2)
:=p\big((\omega_1,\omega_2)g\big)
=p\big(z_1\omega_1-\overline{z_2}\omega_2,z_2\omega_1+\overline{z_1}\omega_2\big)
\end{equation}
for $g=\left(\begin{smallmatrix}z_1&z_2\\-\overline{z_2}&\overline{z_1}
\end{smallmatrix}\right)\in\SU(2)$, $p\in V_\ell$, and $\omega_1,\omega_2\in\C$
defines a $(\ell+1)$-dimensional IUR of $\SU(2)$ on $V_\ell$, and each
finite-dimensional IUR of $\SU(2)$ is unitarily equivalent to an element of the family
$\{\pi^{(\ell)}\}_{\ell\in\N}$ \cite[Prop.~II.1.1 \& Thm.~II.4.1]{Sug90}. A
calculation using the binomial theorem shows that the matrix elements of
$\pi^{(\ell)}$ with respect to the basis $\{p_j\}_{j=0}^\ell$ are given by
\begin{equation}\label{eq_coeff_SU(2)}
\pi_{jk}^{(\ell)}(g)
:=\big\langle\pi^{(\ell)}(g)p_k,p_j\big\rangle_{V_\ell}
=j\;\!!(\ell-j)!\sum_{m=0}^k\sum_{n=0}^{\ell-k}
\begin{pmatrix}
k\\
m
\end{pmatrix}
\begin{pmatrix}
\ell-k\\
n
\end{pmatrix}
z_1^m\;\!\overline{z_1}^{\ell-k-n}z_2^n(-\overline{z_2})^{k-m}\;\!\delta_{j,m+n},
\end{equation}
with $j,k\in\{0,\ldots,\ell\}$ and
$\big(\begin{smallmatrix}\cdot\\\cdot\end{smallmatrix}\big)$ the binomial
coefficients. In particular, in the case of diagonal elements
$
g=
\big(\begin{smallmatrix}
z_1 & 0\\
0 & \overline{z_1}
\end{smallmatrix}\big)
\in\SU(2)
$,
one obtains
\begin{equation}\label{eq_coeff_diagonal}
\pi_{jk}^{(\ell)}
\begin{pmatrix}
z_1 & 0\\
0 & \overline{z_1}
\end{pmatrix}
=j\;\!!(\ell-j)!\;\!z_1^{2j-\ell}\;\!\delta_{j,k}.
\end{equation}

If $\phi\in C^1\big(X,\SU(2)\big)$, Theorem \ref{thm_mixing_D_pi} implies that the
strong limit
$$
D_{\phi,\ell}
:=D_{\phi,\pi^{(\ell)}}
=\slim_{N\to\infty}\frac1N\big[A,\big(U_{\phi,\pi^{(\ell)},j}\big)^N\big]
\big(U_{\phi,\pi^{(\ell)},j}\big)^{-N}
$$
exists and is equal to
$i\big(\d\pi^{(\ell)}\big)_{I_2}\big((P_\phi M_\phi)(\;\!\cdot\;\!)\big)$ with
$I_2:=\left(\begin{smallmatrix}1&0\\0&1\end{smallmatrix}\right)=e_{\;\!\SU(2)}$, and
$
\lim_{N\to\infty}
\big\langle\varphi,\big(U_{\phi,\pi^{(\ell)},j}\big)^N\psi\big\rangle=0
$
for each $\varphi\in\ker(D_{\phi,\ell})^\perp$ and $\psi\in\H^{(\pi^{(\ell)})}_j$.
Therefore, $U_\phi$ is mixing in the subspace
$$
\bigoplus_{\ell\in\N}\;\bigoplus_{j\in\{0,\ldots,\ell\}}\ker(D_{\phi,\ell})^\perp
\subset\H.
$$

Now, since $\SU(2)$ is connected and semisimple, we know from Remark
\ref{remark_semisimple}(b) that if $T_\phi$ uniquely ergodic then $P_\phi M_\phi$ is
zero. So, assuming $T_\phi$ uniquely ergodic would not lead to additional results with
our methods. Therefore, to give a more explicit description of the continuous spectrum
of $U_\phi$, we assume instead that $F_1$ is ergodic. In this case, Remark
\ref{remark_Pf} implies that there exists a constant $\rho_\phi\ge0$ such that
\begin{equation}\label{eq_rho_phi}
\rho_\phi=\big\|\big(P_\phi M_\phi\big)(x)\big\|_{\su(2)}
\quad\hbox{for $\mu_X$-almost every $x\in X$.}
\end{equation}
Moreover, the proof of Theorem 2.4 of \cite{Fra00_2} implies the following useful
result (just replace in the proof of Theorem 2.4 of \cite{Fra00_2} the group $\T$ by
the manifold $X$):

\begin{Lemma}\label{lemma_cohomologous_SU(2)}
Assume that $\phi\in C^1\big(X,\SU(2)\big)$, that $F_1$ is ergodic, and that
$\rho_\phi>0$. Then, $\phi$ is cohomologous to a diagonal cocycle. More precisely, if
$\big(P_\phi M_\phi\big)(x)$ is written as
$$
\big(P_\phi M_\phi\big)(x)
=\begin{pmatrix}
ia(x) & b(x)+ic(x)\\
-b(x)+ic(x) & -ia(x)
\end{pmatrix},
\quad a(x),b(x),c(x)\in\R,
$$
for $\mu_X$-almost every $x\in X$, and if $\zeta:X\to\SU(2)$ is the measurable
function given by
$$
\zeta(x)
:=\begin{cases}
\begin{pmatrix}i\sqrt{\frac{\rho_\phi+a(x)}{2\rho_\phi}}
\frac{b(x)-ic(x)}{|b(x)+ic(x)|}
&\sqrt{\frac{\rho_\phi-a(x)}{2\rho_\phi}}\vspace{2pt}\\
-\sqrt{\frac{\rho_\phi-a(x)}{2\rho_\phi}}
&-i\sqrt{\frac{\rho_\phi+a(x)}{2\rho_\phi}}
\frac{b(x)+ic(x)}{|b(x)+ic(x)|}\end{pmatrix}
& \hbox{if $|a(x)|\ne \rho_\phi$,}\medskip\\
\hfil\begin{pmatrix}0&-1\\1&0\end{pmatrix} & \hbox{if $a(x)=-\rho_\phi$,}\medskip\\
\hfil\begin{pmatrix}1&0\\0&1\end{pmatrix} & \hbox{if $a(x)=\rho_\phi$,}
\end{cases}
$$
then
$
\delta(x):=\zeta(x)\phi(x)(\zeta\circ F_1)(x)^{-1}\in\SU(2)
$
is diagonal for $\mu_X$-almost every $x\in X$.
\end{Lemma}

Remark \ref{remark_cohomologous} and Lemma \ref{lemma_cohomologous_SU(2)} imply that
the Koopman operators $U_\phi$ and $U_\delta$ are unitarily equivalent, and thus have
the same spectral properties. But, since the cocycle $\delta:X\to\SU(2)$ is diagonal,
it is more convenient to work with $U_\delta$. Indeed, since $\delta:X\to\SU(2)$ is
diagonal, \eqref{eq_coeff_diagonal} implies that $\pi^{(\ell)}\circ\delta$ is diagonal
for each $\ell\in\N$. Therefore, we can use Lemma \ref{lemma_diagonal} to determine
the explicit form of the function $P_\delta M_\delta:$

\begin{Lemma}\label{lemma_diagonal_delta}
Assume that $\phi,\zeta\in C^1\big(X,\SU(2)\big)$, that $F_1$ is uniquely ergodic, and
that $\rho_\phi>0$. Then,
$$
\lim_{N\to\infty}\left\|\frac1N\sum_{n=0}^{N-1}
W_{\pi^{(\ell)}\circ\delta}^nM_{\pi^{(\ell)}\circ\delta}
-\big(\d\pi^{(\ell)}\big)_{I_2}(M_{\delta,\star})\right\|_{\linf(X;\B(\C^{\ell+1}))}=0
$$
with $M_{\delta,\star}:=\int_X\d\mu_X(x)\,M_\delta(x)$. Furthermore,
$
P_\delta M_\delta
=M_{\delta,\star}
=\left(\begin{smallmatrix}
i\rho_\phi & 0\\
0 & -i\rho_\phi
\end{smallmatrix}\right)
$
$\mu_X$-almost everywhere.
\end{Lemma}

\begin{proof}
Since $\phi,\zeta\in C^1\big(X,\SU(2)\big)$, we have
$\delta\in C^1\big(X,\SU(2)\big)$. So, we can apply Lemma \ref{lemma_diagonal} with
$\phi$ replaced by $\delta$ to obtain that
$$
P_\delta M_\delta=M_{\delta,\star}=\int_X\d\mu_X(x)\,M_\delta(x)
$$
$\mu_X$-almost everywhere. Now, since $M_\delta(x)$ is diagonal and belongs to
$\su(2)$ for each $x\in X$, there exists $s\in\R$ such that
$$
\int_X\d\mu_X(x)\,M_\delta(x)
=\begin{pmatrix}
is & 0\\
0 & -is
\end{pmatrix}.
$$
Therefore, it only remains to show that $s=\rho_\phi$ to conclude. For this, we note
that \eqref{eq_equal_norms} and \eqref{eq_rho_phi} imply for $\mu_X$-almost every
$x\in X$ that
$$
(\rho_\phi)^2
=\big\|\big(P_\phi M_\phi\big)(x)\big\|_{\su(2)}^2
=\big\|\big(P_\delta M_\delta\big)(x)\big\|_{\su(2)}^2
=\frac12\Tr\left(
\begin{pmatrix}
is & 0\\
0 & -is
\end{pmatrix}
\begin{pmatrix}
is & 0\\
0 & -is
\end{pmatrix}^*
\right)
=s^2.
$$
Since $\rho_\phi>0$, we thus obtain that $s=\rho_\phi$ as desired.
\end{proof}

Equation \eqref{eq_coeff_diagonal} and Lemma \ref{lemma_diagonal_delta} imply that
$$
D_{\delta,\ell}
:=D_{\delta,\pi^{(\ell)}}
=\slim_{N\to\infty}\frac1N\big[A,\big(U_{\delta,\pi^{(\ell)},j}\big)^N\big]
\big(U_{\delta,\pi^{(\ell)},j}\big)^{-N}
=i\big(\d\pi^{(\ell)}\big)_{I_2}
\left(\begin{pmatrix}
i\rho_\phi & 0\\
0 & -i\rho_\phi
\end{pmatrix}\right)
$$
with
$$
i\big(\d\pi^{(\ell)}\big)_{I_2}
\left(\begin{pmatrix}
i\rho_\phi & 0\\
0 & -i\rho_\phi
\end{pmatrix}\right)_{j,k}
=i\;\!\frac\d{\d t}\Big|_{t=0}\;\!\pi^{(\ell)}_{j,k}
\left(\e^{t\left(\begin{smallmatrix}
i\rho_\phi & 0\\
0 & -i\rho_\phi
\end{smallmatrix}\right)}\right)
=j\;\!!(\ell-j)!\;\!\rho_\phi(\ell-2j)\;\!\delta_{j,k}.
$$
So, we deduce from Remark \ref{remark_kernel} that
\begin{align}
\ker(D_{\delta,\ell})^\perp
&=\bigoplus_{k\in\{0,\ldots,\ell\},\,2k\ne\ell}
\ltwo(X,\mu_X)\otimes\big\{\pi^{(\ell)}_{jk}\big\}\nonumber\\
&=\begin{cases}
\displaystyle\bigoplus_{k\in\{0,\ldots,\ell\}\setminus\{\ell/2\}}
\ltwo(X,\mu_X)\otimes\big\{\pi^{(\ell)}_{jk}\big\}
& \hbox{if $\ell\in2\N$,}\\
\hfil\H^{(\pi^{(\ell)})}_j & \hbox{if $\ell\in2\N+1$.}
\end{cases}\label{eq_kernel_delta}
\end{align}

Combining what precedes with Corollary \ref{cor_mixing_D_pi}, we obtain the following
result for the mixing property of the Koopman operator $U_\delta:$

\begin{Theorem}\label{thm_mixing_U_delta}
Assume that $\phi,\zeta\in C^1\big(X,\SU(2)\big)$, that $F_1$ is uniquely ergodic, and
that $\rho_\phi>0$. Then,
$$
\lim_{N\to\infty}
\big\langle\varphi,\big(U_{\delta,\pi^{(\ell)},j}\big)^N\psi\big\rangle=0
\quad\hbox{for each $\varphi\in\ker(D_{\delta,\ell})^\perp$ and
$\psi\in\H^{(\pi^{(\ell)})}_j$,}
$$
with $\ker(D_{\delta,\ell})^\perp$ given by \eqref{eq_kernel_delta}. In particular,
$U_\delta$ is mixing in the subspace
$$
\H_{\rm mix}
:=\left(\bigoplus_{\ell\in2\N}\;\bigoplus_{j\in\{0,\ldots,\ell\}}\;
\bigoplus_{k\in\{0,\ldots,\ell\}\setminus\{\ell/2\}}\ltwo(X,\mu_X)\otimes
\big\{\pi^{(\ell)}_{jk}\big\}\right)
\bigoplus
\left(\bigoplus_{\ell\in2\N+1}\;\bigoplus_{j\in\{0,\ldots,\ell\}}\;
\H^{(\pi^{(\ell)})}_j\right)
\subset\H.
$$
\end{Theorem}

Using the unitary equivalence of $U_\phi$ and $U_\delta$, one deduces from Theorem
\ref{thm_mixing_U_delta} a criterion for the mixing property of the Koopman operator
$U_\phi$ (just apply Lemma \ref{lemma_inv_degrees}(d)). We refer to
\cite[Thm.~5.3]{Fra00_2} for a related result of K. Fr{\c{a}}czek in the particular
case $X=\T$ and
$$
F_t(x):=x\e^{2\pi it\alpha},\quad t\in\R,~x\in\T,~\alpha\in\R\setminus\Q.
$$

Under an additional regularity assumption of Dini-type on the functions
$\L_Y(\pi^{(\ell)}\circ\phi)$ and $\L_Y(\pi^{(\ell)}\circ\zeta)$, one can show that
$U_\delta$ has purely absolutely continuous spectrum in an appropriate subspace of
$\H_{\rm mix}:$

\begin{Theorem}\label{thm_absolute_U_delta}
Assume that $\phi,\zeta\in C^1\big(X,\SU(2)\big)$, that $F_1$ is uniquely ergodic,
that $\rho_\phi>0$, and that
\begin{equation}\label{eq_Dini_1}
\int_0^1\frac{\d t}t\,\big\|\L_Y(\pi^{(\ell)}\circ\phi)_{jk}\circ F_t
-\L_Y(\pi^{(\ell)}\circ\phi)_{jk}\big\|_{\linf(X,\mu_X)}<\infty
\end{equation}
and
\begin{equation}\label{eq_Dini_2}
\int_0^1\frac{\d t}t\,\big\|\L_Y(\pi^{(\ell)}\circ\zeta)_{jk}\circ F_t
-\L_Y(\pi^{(\ell)}\circ\zeta)_{jk}\big\|_{\linf(X,\mu_X)}<\infty
\end{equation}
for each $\ell\in2\N+1$ and $j,k\in\{0,\ldots,\ell\}$. Then, $U_\delta$ has purely
absolutely continuous spectrum in the subspace
$$
\H_{\rm ac}
:=\bigoplus_{\ell\in2\N+1}\;\bigoplus_{j\in\{0,\ldots,\ell\}}\H^{(\pi^{(\ell)})}_j
\subset\H_{\rm mix}.
$$
\end{Theorem}

\begin{proof}
We apply Corollary \ref{cor_absolute_D_pi} with $\phi$ replaced by $\delta$. For this,
we first note that assumptions (i)-(ii) of Corollary \ref{cor_absolute_D_pi} are
trivially satisfied. For assumption (iii), we have for each $\ell\in\N$ and
$j\in\{0,\ldots,\ell\}$
$$
A_{D_{\delta,\ell}}=AD_{\delta,\ell}+D_{\delta,\ell}A=2AD_{\delta,\ell},
$$
with $D_{\delta,\ell}$ the multiplication operator by the constant diagonal matrix
$$
i\big(\d\pi^{(\ell)}\big)_{I_2}
\left(\begin{pmatrix}
i\rho_\phi & 0\\
0 & -i\rho_\phi
\end{pmatrix}\right)
=\rho_\phi
\begin{pmatrix}
0\;\!!(\ell-0)!\;\!(\ell-2\cdot0) &  & 0\\
 & \ddots & \\
0 & & \ell\;\!!(\ell-\ell)!\;\!(\ell-2\cdot\ell)
\end{pmatrix}.
$$
Thus, $C^{+0}(A)\subset C^{+0}\big(A_{D_{\delta,\ell}}\big)$ and it is sufficient to
check that
$$
M_{\pi^{(\ell)}\circ\delta}
=\L_Y(\pi^{(\ell)}\circ\delta)\cdot(\pi^{(\ell)}\circ\delta)^{-1}\in C^{+0}(A).
$$
Now, $\pi^{(\ell)}\circ\delta\in C^1\big(X;\B(\C^{\ell+1})\big)$ since
$\phi,\zeta\in C^1\big(X,\SU(2)\big)$. So,
$(\pi^{(\ell)}\circ\delta)^{-1}\in C^1(A)\subset C^{+0}(A)$ (see
\cite[Prop.~5.1.6(a)]{ABG96}), and we only have to check that
$\L_Y(\pi^{(\ell)}\circ\delta)\in C^{+0}(A)$ (see \cite[Prop.~5.2.3(b)]{ABG96}). But,
since
$$
\pi^{(\ell)}\circ\delta
=(\pi^{(\ell)}\circ\zeta)\cdot(\pi^{(\ell)}\circ\phi)
\cdot(\pi^{(\ell)}\circ\zeta\circ F_1)^{-1},
$$
we have
\begin{align*}
&\L_Y(\pi^{(\ell)}\circ\delta)\\
&=\L_Y(\pi^{(\ell)}\circ\zeta)\cdot(\pi^{(\ell)}\circ\phi)
\cdot(\pi^{(\ell)}\circ\zeta\circ F_1)^{-1}+(\pi^{(\ell)}\circ\zeta)
\cdot\L_Y(\pi^{(\ell)}\circ\phi)\cdot(\pi^{(\ell)}\circ\zeta\circ F_1)^{-1}\\
&\quad-(\pi^{(\ell)}\circ\zeta)\cdot(\pi^{(\ell)}\circ\phi)
\cdot(\pi^{(\ell)}\circ\zeta\circ F_1)^{-1}\cdot\L_Y(\pi^{(\ell)}\circ\zeta\circ F_1)
\cdot(\pi^{(\ell)}\circ\zeta\circ F_1)^{-1}
\end{align*}
with
$
(\pi^{(\ell)}\circ\phi),(\pi^{(\ell)}\circ\zeta),
(\pi^{(\ell)}\circ\zeta\circ F_1)^{-1}\in C^1(A)
$. 
Thus, in the end it is sufficient to check that $\L_Y(\pi^{(\ell)}\circ\zeta)$,
$\L_Y(\pi^{(\ell)}\circ\phi)$ and $\L_Y(\pi^{(\ell)}\circ\zeta\circ F_1)$ belong to
$C^{+0}(A)$, which follows from \eqref{eq_Dini_1}-\eqref{eq_Dini_2}.

Finally, we have for each $\ell\in2\N+1$ and $j\in\{0,\ldots,\ell\}$
\begin{align*}
\det\left(i\big(\d\pi^{(\ell)}\big)_{I_2}(M_{\delta,\star})\right)
&=\det\left(\rho_\phi
\begin{pmatrix}
0\;\!!(\ell-0)!\;\!(\ell-2\cdot0) &  & 0\\
 & \ddots & \\
0 & & \ell\;\!!(\ell-\ell)!\;\!(\ell-2\cdot\ell)
\end{pmatrix}\right)\\
&=(\rho_\phi)^{\ell+1}\prod_{j=0}^\ell j\;\!!(\ell-j)!\;\!(\ell-2j)\\
&\ne0.
\end{align*}
Therefore, the last assumption of Corollary \ref{cor_absolute_D_pi} is satisfied, and
thus the claim follows from Corollary \ref{cor_absolute_D_pi}.
\end{proof}

Using the unitary equivalence of $U_\phi$ and $U_\delta$, one deduces from Theorem
\ref{thm_absolute_U_delta} the absolute continuity of the continuous spectrum of the
Koopman operator $U_\phi$ in the subspace
$$
\H_{\rm ac}
=\bigoplus_{\ell\in2\N+1}\;\bigoplus_{j\in\{0,\ldots,\ell\}}\H^{(\pi^{(\ell)})}_j
\subset\H_{\rm mix}.
$$
This result extends a similar result of K. Fr{\c{a}}czek (see
\cite[Cor.~6.5]{Fra00_2}) in the particular case $X=\T$ and
$$
F_t(x):=x\e^{2\pi it\alpha},\quad t\in\R,~x\in\T,~\alpha\in\R\setminus\Q.
$$

%--------------------------------------------------------------------------------------
\subsection{Cocycles with values in $\SO(3,\R)$}\label{section_SO(3,R)}
%--------------------------------------------------------------------------------------

Let $\GL(3,\R)$ be the set of $3\times3$ invertible real matrices, let
$I_3:=\left(\begin{smallmatrix}1&0&0\\0&1&0\\0&0&1\end{smallmatrix}\right)$, and
assume that
\begin{gather*}
G=\SO(3,\R)=\left\{g\in\GL(3,\R)\mid g^\top g=I_3,~\det(g)=1\right\},\\
\g=\so(3,\R)=\left\{Z\in\GL(3,\R)\mid Z^\top=-Z\right\},\\
\langle Z_1,Z_2\rangle_{\so(3,\R)}:=\frac12\Tr\big(Z_1Z_2^\top\big),
\quad Z_1,Z_2\in\so(3,\R).
\end{gather*}
The set $\widehat{\SO(3,\R)}$ of finite-dimensional IUR's of $\SO(3,\R)$ can be
described as follows. For each $\ell\in\N$, let $W_\ell$ be the
$\frac12(\ell+1)(\ell+2)$-dimensional vector space of homogeneous polinomials of
degree $\ell$ in the variables $x_1,x_2,x_3\in\R$, and let $H_\ell$ be the
$(2\ell+1)$-dimensional vector space of harmonic polynomials of degree $\ell$, that
is,
$$
H_\ell:=\big\{f\in W_\ell\mid\triangle f=0\big\}
$$
with $\triangle$ the Laplacian on $\R^3$. Endow $H_\ell$ with the scalar product
$$
\big\langle f_1,f_2\big\rangle_{H_\ell}
:=\int_{\S^2}\d\omega\,f_1(\omega)\overline{f_2(\omega)},
$$
where $\d\omega$ is the normalised Lebesgue measure on the $2$-sphere $\S^2$ (that is,
$\d\omega=(4\pi)^{-1}\sin\theta\;\!\d\theta\;\!\d\varphi$ in spherical coordinates).
Then, the family $\{\Upsilon_{\ell,j}\}_{j=-\ell}^\ell\subset H_\ell$ of harmonic
polynomials given in spherical coordinates by
$$
\Upsilon_{\ell,j}(r,\theta,\varphi):=r^\ell\e^{ij\varphi}\Theta_{\ell,j}(\theta),
\quad r\ge0,~\theta\in[0,\pi],~\varphi\in[0,2\pi),
$$
with
$$
\begin{cases}
\displaystyle
\Theta_{\ell,j}(\theta)
:=(-1)^j\sqrt{\frac{(2\ell+1)(\ell-j)!}{(\ell+j)!}}\sin^j\theta\;\!
\frac{\d^j}{\d(\cos\theta)^j}\;\!P_\ell(\cos\theta)\smallskip\\
\hfil\Theta_{\ell,-j}(\theta):=(-1)^j\Theta_{\ell,j}(\theta)
\end{cases}
\quad(j\ge0)
$$
and
$$
P_\ell(\cos\theta)
:=\frac1{2^\ell\;\!\ell!}\frac{\d^\ell}{\d(\cos\theta)^\ell}
\big(\cos^2\theta-1\big)^\ell,
$$
is an orthonormal basis of $H_\ell$ \cite[Sec.~B.10]{Mes61}. Furthermore, the function
$\pi^{(\ell)}:\SO(3,\R)\to\U(H_\ell)\simeq\U(2\ell+1)$ given by
$$
\big(\pi^{(\ell)}(g)f\big)(x):=f(g^{-1}x),
\quad g\in\SO(3,\R),~f\in H_\ell,~x\in\R^3,
$$
defines a $(2\ell+1)$-dimensional IUR of $\SO(3,\R)$ on $H_\ell$, and each
finite-dimensional IUR of $\SO(3,\R)$ is unitarily equivalent to an element of the
family $\{\pi^{(\ell)}\}_{\ell\in\N}$ (see \cite[Thm.~II.7.2]{Sug90} and
\cite[Sec.~IX.1-2]{Vil68}). Moreover, if $\SO(3,\R)$ is parameterised by the Euler
angles, that is, if the elements $g\in\SO(3,\R)$ are written as a product of three
rotations
$$
g
=\{\alpha,\beta,\gamma\}
:=\begin{pmatrix}
\cos\alpha & \sin\alpha & 0\\
-\sin\alpha & \cos\alpha & 0\\
0 & 0 & 1
\end{pmatrix}
\begin{pmatrix}
\cos\beta & 0 & -\sin\beta\\
0 & 1 & 0\\
\sin\beta & 0 & \cos\beta
\end{pmatrix}
\begin{pmatrix}
\cos\gamma & \sin\gamma & 0\\
-\sin\gamma & \cos\gamma & 0\\
0 & 0 & 1
\end{pmatrix}
$$
with $\alpha,\gamma\in[0,2\pi]$ and $\beta\in[0,\pi]$, then the matrix elements of
$\pi^{(\ell)}$ with respect to the basis $\{\Upsilon_{\ell,j}\}_{j=-\ell}^\ell$ are
given by \cite[Eq.~15.5 \& 15.27]{Wig59}
\begin{align*}
&\pi_{jk}^{(\ell)}(\{\alpha,\beta,\gamma\})\\
&:=\big\langle\pi^{(\ell)}(\{\alpha,\beta,\gamma\})\Upsilon_{\ell,k},
\Upsilon_{\ell,j}\big\rangle_{H_\ell}\\
&=\sum_m(-1)^m\frac{\sqrt{(\ell+k)!\;\!(\ell-k)!\;\!(\ell+j)!\;\!(\ell-j)!}}
{(\ell-j-m)!\;\!(\ell+k-m)!\;\!m!\;\!(m+j-k)!}\e^{i(j\alpha+k\gamma)}
\cos^{2\ell+k-j-2m}(\beta/2)\sin^{2m+j-k}(\beta/2)
\end{align*}
with $j,k\in\{-\ell,\ldots,\ell\}$ and the sum over the values $m\in\N$ for which the
factorial arguments are equal or greater than zero. In the particular case where $g$
is a rotation of angle $\alpha$ about the $x_3$-axis, that is, $g=\{\alpha,0,0\}$, one
obtains \cite[Eq.~15.6]{Wig59}
\begin{equation}\label{eq_SO(3,R)_diagonal}
\pi_{jk}^{(\ell)}(\{\alpha,0,0\})=\e^{ij\alpha}\delta_{jk}.
\end{equation}

If $\phi\in C^1\big(X,\SO(3,\R)\big)$, Theorem \ref{thm_mixing_D_pi} implies that the
strong limit
$$
D_{\phi,\ell}
:=D_{\phi,\pi^{(\ell)}}
=\slim_{N\to\infty}\frac1N\big[A,\big(U_{\phi,\pi^{(\ell)},j}\big)^N\big]
\big(U_{\phi,\pi^{(\ell)},j}\big)^{-N}
$$
exists and is equal to
$i\big(\d\pi^{(\ell)}\big)_{I_3}\big((P_\phi M_\phi)(\;\!\cdot\;\!)\big)$, and
$
\lim_{N\to\infty}\big\langle\varphi,
\big(U_{\phi,\pi^{(\ell)},j}\big)^N\psi\big\rangle=0
$
for each $\varphi\in\ker(D_{\phi,\ell})^\perp$ and $\psi\in\H^{(\pi^{(\ell)})}_j$.
Therefore, $U_\phi$ is mixing in the subspace
$$
\bigoplus_{\ell\in\N}\;\bigoplus_{j\in\{-\ell,\ldots,\ell\}}
\ker(D_{\phi,\ell})^\perp\subset\H.
$$

Now, since $\SO(3,\R)$ is a connected and semisimple, we know from Remark
\ref{remark_semisimple}(b) that there is no uniquely ergodic skew product $T_\phi$
with $\phi\in C^1\big(X,\SO(3,\R)\big)$ and nonzero degree $P_\phi M_\phi$ (even
though there are $T_\phi$ with $\phi\in C^1\big(X,\SO(3,\R)\big)$ which are uniquely
ergodic, see \cite{Eli02,Hou11,Ner88}). So, assuming $T_\phi$ uniquely ergodic does
not lead to additional results with our methods. On another hand, one can assume that
$F_1$ is uniquely ergodic and $\pi\circ\phi$ diagonal, and then apply Lemma
\ref{lemma_diagonal}, Corollary \ref{cor_mixing_D_pi} and Corollary
\ref{cor_absolute_D_pi} to get a more explicit description of the continuous spectrum
of $U_\phi$. Thanks to \eqref{eq_SO(3,R)_diagonal}, this can be done for example in
the case the cocycle $\phi$ takes values in the maximal torus $T$ given by rotations
about the $x_3$-axis:
$$
T:=\big\{\{\alpha,0,0\}\mid\alpha\in[0,2\pi]\big\}\subset\SO(3,\R).
$$
We leave the details of the calculations in this case to the reader.

%--------------------------------------------------------------------------------------
\subsection{Cocycles with values in $\U(2)$}\label{section_U(2)}
%--------------------------------------------------------------------------------------

As a final example, we consider the case where the cocycle $\phi$ takes values in the
group $\U(2)$. But, we emphasize that one could perfectly go on with examples since
the technics of this article can a priori be applied to cocycles $\phi$ taking values
in any compact Lie group.

Assume that
\begin{gather*}
G
=\U(2)
=\left\{
\begin{pmatrix}
z_1 & z_2\\
-\e^{i\theta}\overline{z_2} & \e^{i\theta}\overline{z_1}
\end{pmatrix}
\mid\theta\in[0,2\pi),~z_1,z_2\in\C,~|z_1|^2+|z_2|^2=1\right\},\\
\g=\fraku(2)
=\left\{
\begin{pmatrix}
is_1 & z\\
-\overline z & is_2
\end{pmatrix}
\mid s_1,s_2\in\R,~z\in\C\right\},\\
\langle Z_1,Z_2\rangle_{\fraku(2)}
:=\frac12\Tr\big(Z_1Z_2^*\big),\quad Z_1,Z_2\in\fraku(2).
\end{gather*}
The set $\widehat{\U(2)}$ of finite-dimensional IUR's of $\U(2)$ can be described as
follows. The map
$$
\T\times\SU(2)\ni(z,g)\mapsto zg\in\U(2)
$$
is an epimorphism with kernel $\{(1,I_2),(-1,-I_2)\}$. Therefore, the elements of
$\widehat{\U(2)}$ are given (up to unitary equivalence) as tensors products
$\{\rho_{2m-\ell}\otimes\pi^{(\ell)}\}_{\ell\in\N,m\in\Z}$, with
$\pi^{(\ell)}:\SU(2)\to\U(V_\ell)\simeq\U(\ell+1)$ the IUR of $\SU(2)$ on $V_\ell$
defined in \eqref{eq_rep_SU(2)} and $\rho_m:\T\to\U(1)=\T$ the IUR of $\S^1$ on $\C$
defined by
$$
\rho_m(z)\;\!\omega:=z^m\omega,\quad z\in\T,~\omega\in\C.
$$
(The indices $2m-\ell$ and $\ell$ add to $2m$ so that $(-1,-I_2)$ belongs to the
kernel of $\rho_{2m-\ell}\otimes\pi^{(\ell)}$, see \cite[Sec.~II.5]{BtD85}). This
implies that the set $\widehat{\U(2)}$ coincides with the family
$\{\pi^{(\ell,m)}\}_{\ell\in\N,m\in\Z}$ of representations
$\pi^{(\ell,m)}:\U(2)\to\U(\C\otimes V_\ell)\simeq\U(\ell+1)$ given by
$$
\pi^{(\ell,m)}(zg):=\rho_{2m-\ell}(z)\otimes\pi^{(\ell)}(g),
\quad z\in\T,~g\in\SU(2).
$$
The matrix elements of $\pi^{(\ell,m)}$ with respect to the basis
$\{1\otimes p_j\}_{j=0}^\ell$ of $\C\otimes V_\ell$ are given by
\begin{align}
\pi^{(\ell,m)}_{jk}(zg)
&:=\big\langle\big(\rho_{2m-\ell}(z)\otimes\pi^{(\ell)}(g)\big)(1\otimes p_k),
1\otimes p_j\big\rangle_{\C\otimes V_\ell}\nonumber\\
&=\big\langle z^{2m-\ell},1\big\rangle_\C\;\!
\big\langle\pi^{(\ell)}(g)p_k,p_j\big\rangle_{V_\ell}\nonumber\\
&=z^{2m-\ell}\;\!\pi_{jk}^{(\ell)}(g),\label{eq_coeff_U(2)}
\end{align}
with $j,k\in\{0,\ldots,\ell\}$, $z\in\T$, $g\in\SU(2)$, and $\pi_{jk}^{(\ell)}(g)$ as
in \eqref{eq_coeff_SU(2)}.

If $\phi\in C^1\big(X;\U(2)\big)$, Theorem \ref{thm_mixing_D_pi} implies that the
strong limit
$$
D_{\phi,\ell,m}
:=D_{\phi,\pi^{(\ell,m)}}
=\slim_{N\to\infty}\frac1N\big[A,\big(U_{\phi,\pi^{(\ell,m)},j}\big)^N\big]
\big(U_{\phi,\pi^{(\ell,m)},j}\big)^{-N}
$$
exists and is equal to
$i\big(\d\pi^{(\ell,m)}\big)_{I_2}\big((P_\phi M_\phi)(\;\!\cdot\;\!)\big)$, and
$
\lim_{N\to\infty}
\big\langle\varphi,\big(U_{\phi,\pi^{(\ell,m)},j}\big)^N\psi\big\rangle=0
$
for each $\varphi\in\ker(D_{\phi,\ell,m})^\perp$ and $\psi\in\H^{(\pi^{(\ell,m)})}_j$.
Therefore, $U_\phi$ is mixing in the subspace
$$
\bigoplus_{m\in\Z}\;\bigoplus_{\ell\in\N}\;\bigoplus_{j\in\{0,\ldots,\ell\}}
\ker(D_{\phi,\ell,m})^\perp\subset\H.
$$

To give a more explicit description of the continuous spectrum of $U_\phi$, we assume
from now on that $T_\phi$ is uniquely ergodic (another possibility would be to assume
$F_1$ uniquely ergodic and $\pi^{(\ell,m)}\circ\phi$ diagonal as in Lemma
\ref{lemma_diagonal}; we refer to \cite[Sec.~4.3]{Tie15_2} for results in that
direction when $X=\T^d$). In this case, Lemma \ref{lemma_ergodic} and Remark
\ref{remark_semisimple}(a) imply that
$D_{\phi,\ell,m}=i\big(\d\pi^{(\ell,m)}\big)_{I_2}(M_{\phi,\star})$ with
$$
M_{\phi,\star}
=P_{\Ad}\left(\int_X\d\mu_X(x)\,M_\phi(x)\right)
\in z\big(\fraku(2)\big),
$$
where $z\big(\fraku(2)\big)=\{isI_2\mid s\in\R\}$ is the center of the Lie algebra
$\fraku(2)$. Therefore, $M_{\phi,\star}=is_\phi I_2$ for some $s_\phi\in\R$, and we
obtain from \eqref{eq_coeff_diagonal} and \eqref{eq_coeff_U(2)} that
\begin{align*}
i\big(\d\pi^{(\ell,m)}\big)_{I_2}(M_{\phi,\star})_{j,k}
&=i\;\!\frac\d{\d t}\Big|_{t=0}\;\!\pi^{(\ell,m)}_{j,k}\big(\e^{its_\phi I_2}\big)\\
&=i\;\!\frac\d{\d t}\Big|_{t=0}\;\!\e^{its_\phi(2m-\ell)}\pi_{jk}^{(\ell)}(I_2)\\
&=-s_\phi\;\!(2m-\ell)\;\!j\;\!!(\ell-j)!\;\!\delta_{j,k}.
\end{align*}
So, we deduce from Remark \ref{remark_kernel} that
\begin{equation}\label{eq_kernel_U(2)}
\ker(D_{\phi,\ell,m})^\perp
=\bigoplus_{k\in\{0,\ldots,\ell\},\,s_\phi(2m-\ell)\ne0}
\ltwo(X,\mu_X)\otimes\big\{\pi^{(\ell,m)}_{jk}\big\}
=\begin{cases}
\H^{(\pi^{(\ell,m)})}_j & \hbox{if $s_\phi\;\!(2m-\ell)\ne0$}\\
\hfil\{0\} & \hbox{if $s_\phi\;\!(2m-\ell)=0$.}
\end{cases}
\end{equation}

Combining what precedes with Corollary \ref{cor_mixing_D_pi}, we obtain the following
result for the mixing property of the Koopman operator $U_\phi:$

\begin{Theorem}\label{thm_mixing_U(2)}
Assume that $\phi\in C^1\big(X,\U(2)\big)$ and that $T_\phi$ is uniquely ergodic. Then,
$$
\lim_{N\to\infty}\big\langle\varphi,
\big(U_{\phi,\pi^{(\ell,m)},j}\big)^N\psi\big\rangle=0
\quad\hbox{for each $\varphi\in\ker(D_{\phi,\ell,m})^\perp$ and
$\psi\in\H^{(\pi^{(\ell)})}_j$,}
$$
with $\ker(D_{\phi,\ell,m})^\perp$ given by \eqref{eq_kernel_U(2)}. In particular, if
$s_\phi\ne0$, then $U_\phi$ is mixing in the subspace
$$
\H_{\rm mix}
:=\bigoplus_{m\in\Z}\;\bigoplus_{\ell\in\N\setminus\{2m\}}\;
\bigoplus_{j\in\{0,\ldots,\ell\}}\H^{(\pi^{(\ell,m)})}_j
\subset\H.
$$
\end{Theorem}

Under an additional regularity assumption of Dini-type on the matrix-valued functions
$\L_Y(\pi^{(\ell,m)}\circ\phi)$, one can show that $U_\phi$ has in fact purely
absolutely continuous spectrum in $\H_{\rm mix}:$

\begin{Theorem}\label{thm_ac_U(2)}
Assume that $\phi\in C^1\big(X,\U(2)\big)$, that $T_\phi$ is uniquely ergodic, that
$s_\phi\ne0$, and that
\begin{equation}\label{eq_cond_pi_lm}
\int_0^1\frac{\d t}t\,\big\|\L_Y(\pi^{(\ell,m)}\circ\phi)_{jk}\circ F_t
-\L_Y(\pi^{(\ell,m)}\circ\phi)_{jk}\big\|_{\linf(X,\mu_X)}<\infty
\end{equation}
for each $m\in\Z$, $\ell\in\N\setminus\{2m\}$ and $j,k\in\{0,\ldots,\ell\}$. Then,
$U_\phi$ has purely absolutely continuous spectrum in the subspace
$$
\H_{\rm mix}
=\bigoplus_{m\in\Z}\;\bigoplus_{\ell\in\N\setminus\{2m\}}\;
\bigoplus_{j\in\{0,\ldots,\ell\}}\H^{(\pi^{(\ell,m)})}_j
\subset\H.
$$
\end{Theorem}

\begin{proof}
The claim follows from an application of Corollary \ref{cor_absolute_D_pi}. Since the
assumptions (i), (ii) and (iv) of Corollary \ref{cor_absolute_D_pi} are easily
verified under our hypotheses, we only explain how to verify the assumption (iv), that
is, the inclusion
$$
M_{\pi^{(\ell,m)}\circ\phi}
=\L_Y(\pi^{(\ell,m)}\circ\phi)\cdot(\pi^{(\ell,m)}\circ\phi)^{-1}
\in C^{+0}\big(A_{D_{\phi,\ell,m}}\big).
$$
Since $D_{\phi,\ell,m}$ is the multiplication operator by the constant diagonal matrix
$i\big(\d\pi^{(\ell,m)}\big)_{I_2}(M_{\phi,\star})$, the operator
$A_{D_{\phi,\ell,m}}$ satisfies
$$
A_{D_{\phi,\ell,m}}
=AD_{\phi,\ell,m}+D_{\phi,\ell,m}A
=2AD_{\phi,\ell,m}.
$$
Thus, $C^{+0}(A)\subset C^{+0}\big(A_{D_{\phi,\ell,m}}\big)$. Furthermore, the
assumption $\phi\in C^1\big(X,\U(2)\big)$ implies that
$(\pi^{(\ell,m)}\circ\phi)^{-1}\in C^1(A)$ (see \cite[Prop.~5.1.6(a)]{ABG96}).
Therefore, the inclusion
$M_{\pi^{(\ell,m)}\circ\phi}\in C^{+0}\big(A_{D_{\phi,\ell,m}}\big)$ is satisfied if
$\L_Y(\pi^{(\ell,m)}\circ\phi)\in C^{+0}(A)$ (see \cite[Prop.~5.2.3(b)]{ABG96}), which
in turn follows from \eqref{eq_cond_pi_lm}.
\end{proof}

In the rest of the section, we present an explicit example of skew-product $T_\phi$
satisfying the assumptions of Theorems \ref{thm_mixing_U(2)} and \ref{thm_ac_U(2)},
namely, a skew-product $T_\phi$ with $\phi\in C^\infty\big(\T^d;\U(2)\big)$
($d\in\N^*$), $T_\phi$ uniquely ergodic, and nonzero degree $M_{\phi,\star}$.

\begin{Example}\label{ex_skew_product_U(2)}
It is known that there exists a Lie group isomorphism
$$
h:\SO(3,\R)\times\T\to\U(2).
$$
Therefore, if we find a cocycle $\delta\in C^\infty\big(\T^d,(\SO(3,\R)\times\T)\big)$
with associated skew product
$$
T_\delta:\T^d\times(\SO(3,\R)\times\T)\to\T^d\times(\SO(3,\R)\times\T),
~~\big(x,(g,z)\big)\mapsto\big(F_1(x),(g,z)\;\!\delta(x)\big),
$$
uniquely ergodic and with $M_{\delta,\star}\ne0$, then the skew product
$$
T_\phi:\T^d\times\U(2)\to\T^d\times\U(2),
~~(x,\widetilde g)\mapsto\big(F_1(x),\widetilde g\;\!\phi(x)\big),
\quad\hbox{with}\quad
\phi:=h\circ\delta\in C^\infty\big(\T^d,\U(2)\big),
$$
will also be uniquely ergodic and with nonzero degree $M_{\phi,\star}\ne0$ (the first
fact is standard and the second fact follows from Lemma
\ref{lemma_inv_degrees}(a)). So, we restrict our attention to the construction of
the skew-product $T_\delta$.

We know from the works of L. H. Eliasson and X. Hou (see \cite[Sec.~1]{Eli02} and
\cite[Thm.~1.2]{Hou11}) that there exist vectors
$\alpha=(\alpha_1,\ldots,\alpha_d)\in\R^d$ satisfying a Diophantine condition and
cocycles $\delta_1\in C^\infty\big(\T^d,\SO(3,\R)\big)$ such that the skew product
$$
T_{\delta_1}:\T^d\times\SO(3,\R)\to\T^d\times\SO(3,\R),
~~(x,g)\mapsto\big(x_1\e^{2\pi i\alpha_1},\ldots,x_d\e^{2\pi i\alpha_d},
g\;\!\delta_1(x)\big),
$$
is uniquely ergodic and admits a $C^\infty$ flow $\{F_{1,t}\}_{t\in\R}$. Using
$\alpha$ and $\delta_1$, we define the cocyle
$$
\delta:\T^d\to\SO(3,\R)\times\T,~~x\mapsto\big(\delta_1(x),x_1\big),
$$
and the skew product
\begin{align*}
T_\delta:\T^d\times\big(\SO(3,\R)\times\T\big)
&\to\T^d\times\big(\SO(3,\R)\times\T\big),\\
\big(x,(g,z)\big)&\mapsto\big(x_1\e^{2\pi i\alpha_1},\ldots,x_d\e^{2\pi i\alpha_d},
(g,z)\;\!\delta(x)\big)\\
&~~~~~=\big(x_1\e^{2\pi i\alpha_1},\ldots,x_d\e^{2\pi i\alpha_d},
\big(g\;\!\delta_1(x),zx_1\big)\big).
\end{align*}
We also define the cocycle
$$
\gamma:\T^d\times\SO(3,\R)\to\T,~~(x,g)\mapsto x_1,
$$
and the skew product
\begin{align*}
T_{\gamma}:\big(\T^d\times\SO(3,\R)\big)\times\T
&\to\big(\T^d\times\SO(3,\R)\big)\times\T,\\
\big((x,g),z\big)&\mapsto\big(T_{\delta_1}(x,g),z\;\!\gamma(x,g)\big)
=\big(\big(x_1\e^{2\pi i\alpha_1},\ldots,x_d\e^{2\pi i\alpha_d},
g\;\!\delta_1(x)\big),zx_1\big),
\end{align*}
and note that
$$
T_\delta\big(x,(g,z)\big)=T_{\gamma}\big((x,g),z\big)
\quad\hbox{for each $x\in\T^d$, $g\in\SO(3,\R)$, $z\in\T$,}
$$
due to the identification
$\T^d\times\big(\SO(3,\R)\times\T\big)=\big(\T^d\times\SO(3,\R)\big)\times\T$. So, it
is sufficient to show that $T_{\gamma}$ is uniquely ergodic and with nonzero degree
$M_{\gamma,\star}\ne0$.

Since $T_{\delta_1}$ is uniquely ergodic and $\pi\circ\gamma$ is scalar for each
representation $\pi\in\widehat\T$, we can apply Lemma \ref{lemma_diagonal} to obtain
for $\big(\mu_{\T^d}\otimes\mu_{\SO(3,\R)}\big)$-almost every
$(x,g)\in\T^d\times\SO(3,\R)$ that
$$
\big(P_\gamma M_\gamma\big)(x,g)
=M_{\gamma,\star}
=\int_{\T^d\times\SO(3,\R)}\d\big(\mu_{\T^d}\otimes\mu_{\SO(3,\R)}\big)(x,g)
\,M_\gamma(x,g)
$$
with $M_\gamma(x,g)$ given by
$$
M_\gamma(x,g)
=\left(\frac\d{\d t}\Big|_{t=0}\gamma\big(F_{1,t}(x)\big)\right)\cdot\gamma(x,g)^{-1}
=\left(\frac\d{\d t}\Big|_{t=0}\;\!x_1\e^{2\pi it\alpha_1}\right)\cdot x_1^{-1}
=2\pi i\alpha_1.
$$
Thus, $M_{\gamma,\star}=2\pi i\alpha_1\ne0$ since $\alpha_1\in\R\setminus\Q$.
Therefore, it only remains to show that $T_{\gamma}$ is uniquely ergodic. For this, we
first note that since $\alpha_1\in\R\setminus\Q$, the results of Section
\ref{section_torus} (see \eqref{eq_ac_torus_2}) imply that $T_\gamma$ has purely
absolutely continuous spectrum in the orthocomplement of the functions depending only
on the variables in $\T^d\times\SO(3,\R)$. Therefore, any eigenvector of $U_\gamma$
with eigenvalue $1$, that is, any function
$
\psi\in\ltwo\big(\big(\T^d\times\SO(3,\R)\big)\times\T,
\mu_{\T^d}\otimes\mu_{\SO(3,\R)}\otimes\mu_\T\big)
$
such that $U_\gamma\;\!\psi=\psi$ must satisfy $\psi=\eta\otimes1$ for some
$\eta\in\ltwo\big(\T^d\times\SO(3,\R),\mu_{\T^d}\otimes\mu_{\SO(3,\R)}\big)$. So,
$$
U_\gamma\;\!\psi=\psi
~\Leftrightarrow~U_\gamma(\eta\otimes1)=\eta\otimes1
~\Leftrightarrow~U_{\delta_1}\eta=\eta,
$$
and thus $\eta$ is an eigenvector of $U_{\delta_1}$ with eigenvalue $1$. It follows
that $\eta$ is constant $\big(\mu_{\T^d}\otimes\mu_{\SO(3,\R)}\big)$-almost everywhere
due to the ergodicity of $T_{\delta_1}$. Therefore, $\psi$ is constant
$\big(\mu_{\T^d}\otimes\mu_{\SO(3,\R)}\otimes\mu_\T\big)$-almost everywhere, and
$T_\gamma$ is ergodic. This implies that $T_\gamma$ is uniquely ergodic because unique
ergodicity and ergodicity are equivalent properties for skew products such as
$T_\gamma$ (see \cite[Thm.~4.21]{EW11}).
\end{Example}

%--------------------------------------------------------------------------------------
%\bibliography{../bibliographie/bibliographie}
%--------------------------------------------------------------------------------------

\def\cprime{$'$} \def\polhk#1{\setbox0=\hbox{#1}{\ooalign{\hidewidth
  \lower1.5ex\hbox{`}\hidewidth\crcr\unhbox0}}}
  \def\polhk#1{\setbox0=\hbox{#1}{\ooalign{\hidewidth
  \lower1.5ex\hbox{`}\hidewidth\crcr\unhbox0}}}
  \def\polhk#1{\setbox0=\hbox{#1}{\ooalign{\hidewidth
  \lower1.5ex\hbox{`}\hidewidth\crcr\unhbox0}}} \def\cprime{$'$}
  \def\cprime{$'$} \def\polhk#1{\setbox0=\hbox{#1}{\ooalign{\hidewidth
  \lower1.5ex\hbox{`}\hidewidth\crcr\unhbox0}}}
  \def\polhk#1{\setbox0=\hbox{#1}{\ooalign{\hidewidth
  \lower1.5ex\hbox{`}\hidewidth\crcr\unhbox0}}} \def\cprime{$'$}
  \def\cprime{$'$} \def\cprime{$'$}

%--------------------------------------------------------------------------------------

\end{document}